\documentclass[10pt]{article}
\usepackage[numbers, sort&compress]{natbib}
\usepackage[utf8]{inputenc}
\usepackage[T1]{fontenc}
\usepackage{hyperref}
\usepackage{url}
\usepackage{nicefrac}
\usepackage{microtype}
\usepackage{algorithm}
\usepackage{algorithmic}
\usepackage{amsmath,amssymb,amsthm}
\usepackage{bbm}
\usepackage{thmtools,thm-restate}
\usepackage{color}
\usepackage{stmaryrd}
\usepackage{multirow}
\usepackage{changepage}
\usepackage{paralist}
\usepackage{graphicx}
\usepackage{etoolbox} 
\usepackage{xspace}
\usepackage[usenames,dvipsnames]{xcolor}
\usepackage{fullpage}

\newcommand{\abs}[1]{\left| #1 \right|}
\newcommand{\set}[1]{\left\{ #1 \right\}}
\newcommand{\inner}[2]{\left\langle #1,\, #2 \right\rangle}
\newcommand{\norm}[1]{\left\| #1 \right\|}
\DeclareMathOperator*{\argmax}{arg\,max}
\DeclareMathOperator*{\argmin}{arg\,min}

\newcommand{\Prob}[1]{{\mathbb{P}\left[{#1}\right]}}
\newcommand{\ProbGiven}[2]{{\mathbb{P}\left[{#1}\ \middle|\ #2\right]}}
\newcommand{\E}[1]{{\mathbb{E}\left[{#1}\right]}}

\newcommand{\expo}[1]{\exp\left( #1 \right)}
\newcommand{\parens}[1]{\left( #1 \right)}
\newcommand{\mathand}[0]{\quad\textrm{and}\quad}
\newcommand{\spn}[1]{\textrm{Span}\parens{#1}}

\newcommand{\R}[0]{\mathbb{R}}
\newcommand{\A}[0]{\mathcal{A}}

\newcommand{\Pcal}[0]{\mathcal{P}}
\newcommand{\G}[0]{\mathcal{G}}

\newcommand{\X}[0]{\mathcal{X}}
\newcommand{\Q}[0]{\mathcal{Q}}
\newcommand{\indicator}[1]{\mathbbm{1}_{#1}}
\newcommand{\normalized}[1]{\frac{#1}{\norm{#1}}}

\newcommand{\Ocal}[0]{\mathcal{O}}
\newcommand{\Ograd}[0]{\mathcal{O}_{\textrm{grad}}}
\newcommand{\Oprox}[0]{\mathcal{O}_{\textrm{prox}}}
\newcommand{\FLHB}[0]{\mathcal{F}_{L,H,B}}
\newcommand{\FLHBs}[0]{\mathcal{F}_{L,H,B,\sigma}}

\newcommand{\prox}[0]{\textrm{prox}}

\newcommand{\removed}[1]{}

\newcommand{\collapse}[1]{\dots}

\newcommand{\blockcomment}[1]{}

\newtheorem{theorem}{Theorem}
\newtheorem{lemma}{Lemma}

\definecolor{darkgreen}{rgb}{0,0.4,0.0}
\newcommand{\mcm}[1]{}
\newcommand{\natinote}[1]{}

\newcommand{\machines}{M}      
\newcommand{\K}{K}  
\newcommand{\ambsgd}{A-MB-SGD\xspace}

\newtoggle{graphs_fig}
\toggletrue{graphs_fig}

\iftoggle{graphs_fig}{
  \usepackage{tikz}
  \usepackage{tkz-graph}
  \tikzstyle{vertex}=[circle, draw, inner sep=0pt, minimum size=6pt]
  \usetikzlibrary{arrows}
  \newcommand{\vertex}{\node[vertex]}
  \usepackage{subcaption}
}

\title{\rule{\textwidth}{2pt}\\{\LARGE\textbf{Graph Oracle Models, Lower Bounds, and Gaps for \\ Parallel Stochastic Optimization}}\\\vspace{-3mm}\rule{\textwidth}{1pt}}
\author{
\begin{minipage}[c]{0.32\textwidth}
\vspace{-3mm}
\centering\normalsize
\textbf{Blake Woodworth}\\
Toyota Technological \\
Institute at Chicago \\ 
\texttt{blake@ttic.edu}
\end{minipage} %
\begin{minipage}[c]{0.32\textwidth}
\vspace{-3mm}
\centering\normalsize
\textbf{Jialei Wang}\\
University of Chicago \\ 
\texttt{jialei@uchicago.edu}
\end{minipage} %
\begin{minipage}[c]{0.32\textwidth}
\vspace{-3mm}
\centering\normalsize
\textbf{Adam Smith}\\
Boston University \\ 
\texttt{ads22@bu.edu}
\end{minipage}
\\
\begin{minipage}[c]{0.5\textwidth}
\vspace{3mm}
\centering\normalsize
\textbf{Brendan McMahan}\\
Google \\ 
\texttt{mcmahan@google.com}
\end{minipage} %
\begin{minipage}[c]{0.5\textwidth}
\vspace{3mm}
\centering\normalsize
\textbf{Nathan Srebro}\\
Toyota Technological\\
Institute at Chicago
\footnote{Part of this work was completed while visiting Google} \\ 
\texttt{nati@ttic.edu}
\end{minipage}
}
\date{\vspace{-10mm}}

\begin{document}
\maketitle

\begin{abstract}
  We suggest a general oracle-based framework that captures different parallel
  stochastic optimization settings
  described by a dependency graph, and derive generic lower bounds 
  in terms of this graph.  We then use the framework and derive lower
  bounds for several specific parallel optimization settings,
  including delayed updates and parallel processing with intermittent
  communication.  We highlight gaps between lower and upper bounds on
  the oracle complexity, and cases where the ``natural'' algorithms
  are not known to be optimal.
\end{abstract}

\section{Introduction}
Recently, there has been great interest in stochastic optimization and learning
algorithms that leverage 
parallelism, including e.g.~delayed updates arising from
pipelining and asynchronous concurrent processing, 
synchronous single-instruction-multiple-data
parallelism, and parallelism across distant devices. With the abundance of parallelization 
settings and associated algorithms, it is important 
to precisely formulate the problem, which allows us to ask questions
such as ``is there a better method for this problem than what we have?'' and ``what is the best
we could possibly expect?''

Oracle models have long been a useful framework for formalizing stochastic optimization 
and learning problems. In an oracle model, we place limits on the algorithm's
 access to the optimization objective, but not what it
may do with the information it receives. This allows us to obtain sharp lower bounds, which
can be used to argue that an algorithm is optimal and to identify
 gaps between current algorithms and what might be possible.
Finding such gaps can be very useful---for example, the gap between 
the first order optimization lower bound of \citet{nemirovskii1983problem}
and the best known algorithms at the time inspired Nesterov's accelerated
gradient descent algorithm \cite{nesterov1983method}.

We propose an oracle framework for formalizing 
different parallel optimization problems. We specify the structure of 
parallel computation using an ``oracle graph'' which indicates how an
algorithm accesses the oracle. Each node in the graph corresponds to 
a single stochastic oracle query, and that query (e.g.~the point at 
which a gradient is calculated) must be computed using only oracle accesses in 
ancestors of the node. We generally think of each
stochastic oracle access as being based on a single data sample,
thus involving one or maybe a small number of vector operations. 

In Section \ref{sec:graphLBs} we devise generic lower bounds for
parallel optimization problems in terms of simple properties of the associated
oracle graph, namely the length of the longest dependency chain and the total number of
nodes. In Section \ref{sec:examples}
we study specific parallel optimization settings in which 
many algorithms have been proposed, formulate them as graph-based oracle
parallel optimization problems, instantiate our lower bounds, 
and compare them with the performance guarantees of specific algorithms.
We highlight gaps between the lower
bound and the best known upper bound and also situations where we can
devise an optimal algorithm that matches the lower bound, but where 
this is not the ``natural'' and typical algorithm used in this settings. The latter 
indicates either a gap in our understanding of the ``natural'' algorithm or a need to
depart from it.

\paragraph{Previously suggested models} 
Previous work studied communication lower bounds for parallel convex optimization where
there are $\machines$ machines each containing
a local function (e.g.~a collection of samples from a distribution). Each
machine can perform computation on its own function, and then 
periodically every machine is allowed to transmit information to 
the others. In order to prove meaningful lower bounds based on the number of rounds 
of communication, it is necessary to prevent the machines from simply transmitting
their local function to a central machine, or else any objective could be optimized
in one round. There are two established ways of doing this. First,
one can allow arbitrary computation on the local machines, but restrict the number
of bits that can be transmitted in each round. There is work focusing on specific 
statistical estimation problems that establishes communication lower 
bounds via information-theoretic arguments \citep{zhang2013information,garg2014communication,braverman2016communication}. Alternatively, one can allow the machines to communicate 
real-valued vectors, but restrict the types of computation they are allowed to perform.
For instance, \citet{arjevani2015communication} present communication complexity lower bounds for
algorithms which can only compute vectors that lie in a certain subspace, which includes
e.g.~linear combinations of gradients of their local function. \citet{lee2017distributed} 
assume a similar restriction, but allow the data defining the local functions
to be allocated to the different machines in a strategic manner.
Our framework applies to general stochastic optimization problems and does not impose any restrictions on what computation the
algorithm may perform, and is thus a more direct generalization of the
oracle model of optimization. 

Recently, \citet{duchi18a} considered first-order optimization
in a special case of our
proposed model (the ``simple parallelism'' graph of
Section~\ref{sec:layered}), but their bounds apply in a more
limited parameter regime, see Section~\ref{sec:graphLBs} for
discussion. 


\vspace{-0.2 cm}
\section{The graph-based oracle model}
\label{sec:oracle}
\vspace{-0.2 cm}

We consider the following stochastic optimization problem
\begin{equation}\label{eq:mainproblem}
\min_{x\in\R^m:\norm{x}\leq B} F(x) := \mathbb{E}_{z\sim\Pcal}\left[ f(x;z) \right]
\end{equation}
The problem \eqref{eq:mainproblem} captures many important tasks, such as supervised learning, in which case $f(x;z)$ is the loss of a model parametrized by $x$ on data instance $z$ and the goal is to minimize the population risk $\E{f(x;z)}$. We assume that $f(\cdot;z)$ is convex, $L$-Lipschitz, and $H$-smooth for all $z$. We also allow $f$ to be non-smooth, which corresponds to $H=\infty$. A function $g$ is $L$-Lipschitz when $\norm{g(x)-g(y)} \leq L\norm{x-y}$ for all $x,y$, and it is $H$-smooth when it is differentiable and its gradient is $H$-Lipschitz. 
We consider optimization algorithms that use either a stochastic gradient or stochastic prox oracle ($\Ograd$ and $\Oprox$ respectively):
\begin{align}
\Ograd(x,z) &= \parens{f(x;z),\ \nabla f(x;z)} \label{eq:gradoracle}\\
\Oprox(x,\beta,z) &= \parens{f(x;z),\ \nabla f(x;z),\ \prox_{f(\cdot;z)}(x,\beta)} \label{eq:proxoracle}\\
\textrm{where}\quad \prox_{f(\cdot;z)}(x,\beta) &= \argmin_{y} f(y;z) + \frac{\beta}{2}\norm{y-x}^2
\end{align}
The prox oracle is quite powerful and provides global rather than local information about $f$. In particular, querying the prox oracle with $\beta=0$ fully optimizes $f(\cdot;z)$.

As stated, $z$ is an argument to the oracle, however there are two distinct cases. In the ``fully stochastic'' oracle setting, the algorithm receives an oracle answer corresponding to a random $z \sim \Pcal$. We also consider a setting in which the algorithm is allowed to ``actively query'' the oracle. In this case, the algorithm may either sample $z \sim \Pcal$ or choose a desired $z$ and receive an oracle answer for that $z$. Our lower bounds hold for either type of oracle. Most optimization algorithms only use the fully stochastic oracle, but some require more powerful active queries. 

We capture the structure of a parallel optimization algorithm with a directed, acyclic \textbf{oracle graph} $\G$. Its depth, $D$, is the length of the longest directed path, and the size, $N$, is the number of nodes. Each node in the graph represents a single stochastic oracle access, and the edges in the graph indicate where the results of that oracle access may be used: only the oracle accesses from \emph{ancestors} of each node are available when issuing a new query. These limitations might arise e.g.~due to parallel computation delays or the expense of communicating between disparate machines.

Let $\Q$ be the set of possible oracle queries, with the exact form of queries (e.g., $q=x$ vs.~$q=(x, \beta, z)$) depending on the context. Formally, a randomized optimization algorithm that accesses the stochastic oracle $\Ocal$ as prescribed by the graph $\G$ is specified by associating with each node $v_t$ a query rule $R_t:(\Q,\Ocal(\Q))^*\times \Xi \rightarrow \Q$, plus a single output rule $\hat{X}:(\Q,\Ocal(\Q))^*\times \Xi \rightarrow\X$. We grant all of the nodes access to a source of shared randomness $\xi \in \Xi$ (e.g.~an infinite stream of random bits). The mapping $R_t$ selects a query $q_t$ to make at node $v_t$ using the set of queries and oracle responses in ancestors of $v_t$, namely
\begin{equation}
q_t = R_t\big(\parens{q_i, \Ocal(q_i)\,:\,i \in \textrm{Ancestors}(v_t)}, \xi \big)
\end{equation}
Similarly, the output rule $\hat{X}$ maps from all of the queries and oracle responses to the algorithm's output as $\hat{x} = \hat{X}\parens{(q_i, \Ocal(q_i):i \in [N]), \xi}$.
The essential question is: for a class of optimization problems $(\G,\Ocal,\mathcal{F})$ specified by a dependency graph $\G$, a stochastic oracle $\Ocal$, and a function class $\mathcal{F}$, what is the best possible guarantee on the expected suboptimality of an algorithm's output, i.e.~
\begin{equation}
\inf_{\parens{R_1,\dots,R_N,\hat{X}}} \sup_{f \in \mathcal{F}} 
\mathbb{E}_{\hat{x},z}\left[f(\hat{x};z)\right] - \min_{x} \mathbb{E}_{z}\left[f(x;z)\right]
\end{equation}

In this paper, we consider optimization problems $(\G,\Ocal,\FLHB)$
where $\FLHB$ is the class of convex, $L$-Lipschitz, and $H$-smooth
functions on the domain $\set{x\in\R^m:\norm{x}\leq B}$ and
parametrized by $z$, and $\Ocal$ is either a stochastic gradient
oracle $\Ograd$ \eqref{eq:gradoracle} or a stochastic prox oracle
$\Oprox$ \eqref{eq:proxoracle}. We consider this function class to
contain Lipschitz but non-smooth functions too, which corresponds to
$H = \infty$. Our function class does not bound the dimension $m$ of
the problem, as we seek to understand the best possible guarantees in
terms of Lipschitz and smoothness constants that hold in any
dimension. Indeed, there are (typically impractical) algorithms such
as center-of-mass methods, which might use the dimension in order to
significantly reduce the oracle complexity, but at a potentially huge
computational cost. \citet{nemirovski1994parallel} studied non-smooth
optimization in the case that the dimension is bounded, proving lower
bounds in this setting that scale with the $1/3$-power of the
dimension but have only logarithmic dependence on the
suboptimality. 
%
We do not analyze strongly convex functions, but the situation is similar and lower bounds can be established via reduction \cite{woodworth2016tight}.

\vspace{-0.2 cm}
\section{Lower bounds}\label{sec:graphLBs}
\vspace{-0.2 cm}

We now provide lower bounds for optimization problems $(\G,\Ograd,\FLHB)$ and $(\G,\Oprox,\FLHB)$ in terms of $L$, $H$, $B$, and the depth and size of $\G$. 

\begin{restatable}{theorem}{gradonlyLB} \label{thm:gradonlyLB}
Let $L,B \in (0,\infty)$, $H \in [0,\infty]$, $N \geq D \geq 1$, let $\G$ be any oracle graph of depth $D$ and size $N$ and consider the optimization problem $(\G,\Ograd,\FLHB)$. For any randomized algorithm $\A = (R_1,\dots,R_N,\hat{X})$, there exists a distribution $\Pcal$ and a convex, $L$-Lipschitz, and $H$-smooth function $f$ on a $B$-bounded domain in $\R^m$ for $m = O\parens{\max\set{N^2,D^3N}\log\parens{DN}}$ such that
\[
\mathbb{E}_{\substack{z\sim\Pcal\\\hat{X} \sim \A}}\left[f(\hat{X};z)\right] - \min_{x} \mathbb{E}_{z\sim\Pcal}\left[f(x;z)\right]
\geq 
\Omega\parens{\min\set{\frac{LB}{\sqrt{D}},\ \frac{HB^2}{D^2}} + \frac{LB}{\sqrt{N}}}
\]
\end{restatable}

\begin{restatable}{theorem}{proxLB} \label{thm:proxLB}
Let $L,B \in (0,\infty)$, $H \in [0,\infty]$, $N \geq D \geq 1$, let $\G$ be any oracle graph of depth $D$ and size $N$ and consider the optimization problem $(\G,\Oprox,\FLHB)$. For any randomized algorithm $\A = (R_1,\dots,R_N,\hat{X})$, there exists a distribution $\Pcal$ and a convex, $L$-Lipschitz, and $H$-smooth function $f$ on a $B$-bounded domain in $\R^m$ for $m = O\parens{\max\set{N^2,D^3N}\log\parens{DN}}$ such that
\[
\mathbb{E}_{\substack{z\sim\Pcal\\\hat{X} \sim \A}}\left[f(\hat{X};z)\right] - \min_{x} \mathbb{E}_{z\sim\Pcal}\left[f(x;z)\right] 
\geq \Omega\parens{\min\set{\frac{LB}{D},\ \frac{HB^2}{D^2}} + \frac{LB}{\sqrt{N}}}
\]
\end{restatable}

These are the tightest possible lower bounds in terms of just the depth and size of $\G$ in the sense that for all $D,N$ there are graphs $\G$ and associated algorithms which match the lower bound. Of course, for specific, mostly degenerate graphs they might not be tight. For instance, our lower bound for the graph consisting of a short sequential chain plus a very large number of disconnected nodes might be quite loose due to the artificial inflation of $N$. Nevertheless, for many interesting graphs they are tight, as we shall see in Section \ref{sec:examples}.

Each lower bound has two components: an ``optimization'' term and a
``statistical'' term. The statistical term $\Omega(LB/\sqrt{N})$ is
well known, although we include a brief proof of this portion of the bound
in Appendix \ref{sec:statisticalLB} for completeness. The optimization
term depends on the depth $D$, and indicates, intuitively, the best
suboptimality guarantee that can be achieved by an algorithm using
unlimited parallelism but only $D$ rounds of communication.
\citet{arjevani2015communication} also obtain lower bounds in terms of
rounds of communication, which are similar to how our lower bounds
depend on depth.  However they restricted the type of computations
that are allowed to the algorithm to a specific class of operations,
while we only limit the number of oracle queries and the dependency
structure between them, but allow forming the queries in any arbitrary
way.

Similar to \citet{arjevani2015communication}, to establish the
optimization term in the lower bounds, we construct functions that
require multiple rounds of sequential oracle accesses to optimize. In
the gradient oracle case, we use a single, deterministic function which
resembles a standard construction for first order optimization lower
bounds. 
For the prox case, we construct two functions inspired by previous
lower bounds for round-based and finite sum optimization
\cite{arjevani2015communication,woodworth2016tight}. In order to
account for randomized algorithms that might leave the span of
gradients or proxs returned by the oracle, we use a technique that was
proposed by \citet{woodworth2017lower,woodworth2016tight} and refined
by \citet{carmon2017lower}. For our specific setting, we must slightly
modify existing analysis, which is detailed
in Appendix \ref{sec:vectorguessing}.

A useful feature of our lower bounds is that they apply
when both the Lipschitz constant and smoothness are bounded
\emph{concurrently}. Consequently, 
``non-smooth'' in the subsequent discussion can be read as simply 
identifying the case where the $L$ term 
achieves the minimum as opposed to the $H$ term (even if $H <
\infty$). This is particularly important when studying stochastic parallel
optimization, since obtaining non-trivial guarantees in a purely
stochastic setting requires some sort of control on the magnitude of
the gradients (smoothness by itself is not sufficient), while
obtaining parallelization speedups often requires smoothness, and so
we would like to ask what is the best that can be done when both
Lipschitz and smoothness are controlled.  Interestingly, the dependence 
on both $L$ and $H$ in our bounds is tight, even when the other is constrained,
which shows that the optimization term cannot be substantially reduced
by using both conditions together.

In the case of the gradient oracle, we ``smooth out'' a standard
non-smooth lower bound construction \cite{nemirovskii1983problem,woodworth2017lower}; previous work has used a similar approach in slightly different settings \cite{agarwal2017lower,guzman2015lower}. 
For $\ell \leq L$ and $\eta \leq H$, and orthonormal
$v_1,\dots,v_{D+1}$ drawn uniformly at random, we define the
$\ell$-Lipschitz but non-smooth function $\tilde{f}$, and its
$\ell$-Lipschitz, $\eta$-smooth ``$\eta$-Moreau envelope''
\cite{bauschke2017convex}:
\begin{equation}
\tilde{f}(x) = \max_{1\leq r \leq D+1} \ell \parens{v_r^\top x - \frac{r-1}{2(D+1)^{1.5}}}
\qquad\qquad
f(x) = \min_{y} \tilde{f}(y) + \frac{\eta}{2}\norm{y-x}^2
\end{equation}
This defines a distribution over $f$'s based on the randomness in the draw of $v_1,\dots,v_{D+1}$, and we apply Yao's minimax principle. In Appendix \ref{sec:gradonlyappendix}, we prove Theorem \ref{thm:gradonlyLB} using this construction.

In the case of the prox oracle, we ``straighten out'' the smooth construction of \citet{woodworth2016tight}. For fixed constants $c,\gamma$, we define the following Lipschitz and smooth scalar function $\phi_c$:

\begin{minipage}[c]{0.33\textwidth}
\includegraphics[height=24mm,width=48mm]{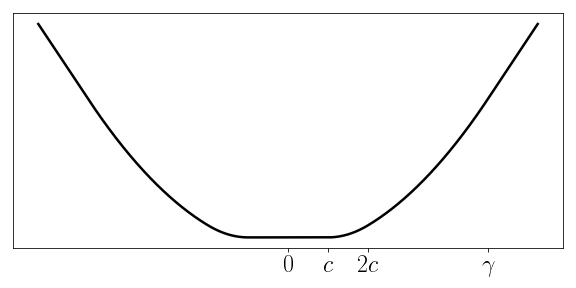}
\end{minipage} %
\begin{minipage}[c]{0.65\textwidth}
\vspace{-5mm}
\begin{align}
\phi_c(z) &= 
\begin{cases} 
0 & \abs{z} \leq c \\
2(\abs{z} - c)^2 & c < \abs{z} \leq 2c \\
z^2 - 2c^2 & 2c < \abs{z} \leq \gamma \\
2\gamma\abs{z} - \gamma^2 - 2c^2 & \abs{z} > \gamma
\end{cases} 
\end{align}
\end{minipage}

For $\Pcal = \textrm{Uniform}\set{1,2}$ and orthonormal $v_1,\dots,v_{2D}$ drawn uniformly at random, we define
\begin{align}
f(x;1) &=  
\frac{\eta}{8}\parens{-2av_{1}^\top x + \phi_c\parens{v_{2D}^\top x} + \hspace{-3mm}\sum_{r=3,5,7,...}^{2D-1} \hspace{-3mm}\phi_c\parens{v_{r-1}^\top x - v_{r}^\top x}}
\\
f(x;2) &= 
\frac{\eta}{8}\parens{\sum_{r=2,4,6,...}^{2D} \hspace{-3mm}\phi_c\parens{v_{r-1}^\top x - v_{r}^\top x}} 
\end{align}
Again, this defines a distribution over $f$'s based on the randomness in the draw of $v_1,\dots,v_{2D}$ and we apply Yao's minimax principle. In Appendix \ref{sec:proxappendix}, we prove Theorem \ref{thm:proxLB} using this construction.

\paragraph{Relation to previous bounds}
As mentioned above, \citet{duchi18a} recently showed a lower bound for
first- and zero-order stochastic optimization in the ``simple
parallelism'' graph consisting of $D$ layers, each with $M$
nodes. Their bound \citep[Thm 2]{duchi18a} applies only when
the dimension $m$ is constant, and $D = O(m\log \log M)$. Our
lower bound requires non-constant dimension, but applies in any range
of $M$. Furthermore, their proof techniques do not obviously
extend to prox oracles.

\section{Specific dependency graphs}
\label{sec:examples}
\iftoggle{graphs_fig}{
}

We now use our framework to study four specific parallelization
structures.  The main results (tight complexities and gaps between lower and upper
bounds) are summarized in Table \ref{table:example}. For
simplicity and without loss of generality, we set $B = 1$, i.e.~we normalize the
optimization domain to be $\{x \in \mathbb{R}^m: \norm{x} \leq 1\}$.  
All stated upper and lower bounds are for the expected
suboptimality $\mathbb{E}[F(\hat{x})] - F(x^*)$  of the algorithm's output.

\begin{table}[t]
\begin{center}
\begin{tabular}
{|c|c|c|}\hline
\multirow{1}{*}{ Graph example}   &With gradient oracle  &With gradient and prox oracle \\\hline
  \begin{tabular}{@{}c@{}}path($T$) \\ (Section \ref{sec:path})\end{tabular}  &\multicolumn{2}{|c|}{$\frac{L}{\sqrt{T}}$} \\ \hline
  \begin{tabular}{@{}c@{}}layer($T,M$) \\ (Section \ref{sec:layered})\end{tabular}  & $  \left(    \frac{L }{\sqrt{T}} \!\wedge\! \frac{H }{T^2} \right) \!+\! \frac{L}{\sqrt{MT}}$ &  $\left(  {\color{darkgreen} \frac{L}{T} } \!\wedge\! \frac{H }{T^2}     \right)\!+\! \frac{L}{\sqrt{MT}}$ \\ \hline
  \begin{tabular}{@{}c@{}}delay($T,\tau$) \\ (Section \ref{sec:delay})\end{tabular}   &{\color{darkgreen} $  \left(   \frac{L }{\sqrt{T/\tau}} \!\wedge\! \frac{ H   \tau^2}{T^2}   \right) \!+\! \frac{L}{\sqrt{T}}$} & {\color{darkgreen}$\left(  \frac{L\tau}{T} \!\wedge\! \frac{H\tau^2}{T^2}  \right) \!+\! \frac{L}{\sqrt{T}}$} \\ \hline
    \begin{tabular}{@{}c@{}}intermittent($T,K,M$) \\ (Section \ref{sec:intermittent})\end{tabular}   &  \begin{tabular}{@{}c@{}c@{}}{\color{red} $ {\textstyle \left(    \frac{L }{\sqrt{KT}} \!\wedge\! \frac{H }{K^2T^2} \right) \! + \! \frac{L}{\sqrt{MKT}}}$} \\ {{\color{blue} $  {\textstyle \frac{L}{\sqrt{KT}} \!\wedge\!  \left( \frac{H}{T^2} \!+\! \frac{L}{\sqrt{MKT}} \right)}$} } \\
{\color{blue}$\wedge   \left( \frac{H}{TK} \!+\! \frac{L}{\sqrt{MKT}} \right) \log \left( \frac{MKT}{L} \right) $}     \end{tabular} & \begin{tabular}{@{}c@{}c@{}} {\color{red}$ {\textstyle\left(  \frac{L}{KT} \!\wedge\! \frac{H }{K^2T^2}  \right) \!+\! \frac{L}{\sqrt{MKT}}}$} \\ { {\color{blue}$ {\textstyle \frac{L}{\sqrt{KT}} \!\wedge\!  \left( \left(  \frac{L}{T} \!\wedge\! \frac{H}{T^2}  \right) \!+\! \frac{L}{\sqrt{MKT}}  \right)  }$} } \\
{\color{blue}$\wedge  \left( \frac{H}{TK} \!+\! \frac{L}{\sqrt{MKT}} \right) \log \left( \frac{MKT}{L} \right) $} \end{tabular}  \\ \hline
\end{tabular}
\vspace{0.2 cm}
\caption{\small Summary of upper and lower bounds for stochastic convex optimization of $L$-Lipschitz and $H$-smooth  functions with $T$ iterations, $\machines$ machines, and $\K$ sequential steps per machine.
 {\color{darkgreen} Green} indicates lower bounds matched only by "unnatural" methods, {\color{red} red} and {\color{blue} blue} indicates a gap between the {\color{red} lower} and {\color{blue} upper} bounds.}
\label{table:example}
\vspace{-0.8cm}
\end{center}
\end{table}

\begin{minipage}[c]{0.5\textwidth}
\subsection{Sequential computation: the path graph}\label{sec:path}
\vspace{-1.2mm}
\end{minipage} %
\begin{minipage}[c]{0.5\textwidth}
\hfill
    \begin{tikzpicture}[scale=1.0]
        \vertex(p0) at (-3,0) {};
        \vertex(p1) at (-2,0) {};
        \vertex(p2) at (-1,0) {};
        \vertex(p3) at (0,0) {};
        \vertex(p4) at (1,0) {};
        \vertex(p5) at (2,0) {};
        \vertex(p6) at (3,0) {};
    \tikzset{EdgeStyle/.style={->}}
        \Edge(p0)(p1)
        \Edge(p1)(p2)
        \Edge(p2)(p3)
        \Edge(p3)(p4)
        \Edge(p4)(p5)
        \Edge(p5)(p6);
    \end{tikzpicture}
\end{minipage}


We begin with the simplest model, that of sequential computation
captured by the path graph of length $T$ depicted above. 
The ancestors of each
vertex $v_i,\,i=1 \dots T$ are all the preceding vertices
$(v_1,\ldots,v_{i-1})$.  The sequential model is of course well
studied and understood.  To see how it fits into our framework: A path
graph of length $T$ has a depth of $D=T$ and size of $N=T$, thus with
either gradient or prox oracles, the statistical term is dominant in
Theorems \ref{thm:gradonlyLB} and \ref{thm:proxLB}. These lower bounds are matched by
sequential stochastic gradient descent, yielding a tight complexity of
$\Theta(L/\sqrt{T})$ and the familiar conclusion that SGD is (worst
case) optimal in this setting.

\begin{minipage}[c]{0.5\textwidth}
\subsection{Simple parallelism: the layer graph}
\label{sec:layered}
\vspace{-1.2mm}
\end{minipage} %
\begin{minipage}[c]{0.5\textwidth}
\hfill
   \begin{tikzpicture}[scale=1.0]
        \vertex(p00) at (-3,-0.6) {};
        \vertex(p01) at (-2,-0.6) {};
        \vertex(p02) at (-1,-0.6) {};
        \vertex(p03) at (0,-0.6) {};
        \vertex(p04) at (1,-0.6) {};
        \vertex(p05) at (2,-0.6) {};
        \vertex(p06) at (3,-0.6) {};
        \vertex(p10) at (-3,0) {};
        \vertex(p11) at (-2,0) {};
        \vertex(p12) at (-1,0) {};
        \vertex(p13) at (0,0) {};
        \vertex(p14) at (1,0) {};
        \vertex(p15) at (2,0) {};
        \vertex(p16) at (3,0) {};
        \vertex(p20) at (-3,0.6) {};
        \vertex(p21) at (-2,0.6) {};
        \vertex(p22) at (-1,0.6) {};
        \vertex(p23) at (0,0.6) {};
        \vertex(p24) at (1,0.6) {};
        \vertex(p25) at (2,0.6) {};
        \vertex(p26) at (3,0.6) {};
    \tikzset{EdgeStyle/.style={->}}
        \Edge(p00)(p01)
        \Edge(p00)(p11)
        \Edge(p00)(p21)
        \Edge(p10)(p01)
        \Edge(p10)(p11)
        \Edge(p10)(p21)
        \Edge(p20)(p01)
        \Edge(p20)(p11)
        \Edge(p20)(p21)
        \Edge(p01)(p02)
        \Edge(p01)(p12)
        \Edge(p01)(p22)
        \Edge(p11)(p02)
        \Edge(p11)(p12)
        \Edge(p11)(p22)
        \Edge(p21)(p02)
        \Edge(p21)(p12)
        \Edge(p21)(p22)
        \Edge(p02)(p03)
        \Edge(p02)(p13)
        \Edge(p02)(p23)
        \Edge(p12)(p03)
        \Edge(p12)(p13)
        \Edge(p12)(p23)
        \Edge(p22)(p03)
        \Edge(p22)(p13)
        \Edge(p22)(p23)
        \Edge(p03)(p04)
        \Edge(p03)(p14)
        \Edge(p03)(p24)
        \Edge(p13)(p04)
        \Edge(p13)(p14)
        \Edge(p13)(p24)
        \Edge(p23)(p04)
        \Edge(p23)(p14)
        \Edge(p23)(p24)
        \Edge(p04)(p05)
        \Edge(p04)(p15)
        \Edge(p04)(p25)
        \Edge(p14)(p05)
        \Edge(p14)(p15)
        \Edge(p14)(p25)
        \Edge(p24)(p05)
        \Edge(p24)(p15)
        \Edge(p24)(p25)
        \Edge(p05)(p06)
        \Edge(p05)(p16)
        \Edge(p05)(p26)
        \Edge(p15)(p06)
        \Edge(p15)(p16)
        \Edge(p15)(p26)
        \Edge(p25)(p06)
        \Edge(p25)(p16)
        \Edge(p25)(p26);
    \end{tikzpicture}
\end{minipage}


We now turn to a model in which $\machines$ oracle queries can be made in
parallel, and the results are broadcast for use in making
the next batch of $\machines$ queries. This corresponds to synchronized
parallelism and fast communication between processors. The model is
captured by a layer graph of width $\machines$, depicted above for $M\!=\!3$. 
The graph consists of $T$ layers $i=1,\ldots,T$ each with $\machines$
nodes $v_{t,1},\ldots,v_{t,m}$ whose ancestors include $v_{t',i}$ for all $t' < t$ and
$i \in [\machines]$. The graph has a depth of $D=T$ and size
of $N= \machines T$.  With a stochastic gradient oracle, Theorem
\ref{thm:gradonlyLB} yields a lower bound of:
\begin{equation}\label{eq:LayeredGradLB}
\Omega \left( \min
    \left\{ \frac{L}{\sqrt{T}}, \frac{H}{T^2} \right\}
    + \frac{L}{\sqrt{\machines T}} \right)
\end{equation}
which is matched by accelerated mini-batch SGD (\ambsgd)
\citep{lan2012optimal,cotter2011better}, establishing the optimality
of \ambsgd in this setting.  For sufficiently smooth objectives, the
same algorithm is also optimal even if prox access is allowed, since
Theorem \ref{thm:proxLB}  implies a lower bound of:
\begin{equation}
  \label{eq:LayeredProxLB}
   \Omega \left( \min
    \left\{ \frac{L}{T}, \frac{H}{T^2} \right \}
    + \frac{L}{\sqrt{MT}} \right).
\end{equation}
That is, for smooth objectives, having access to a prox oracle does
not improve the optimal complexity over just using gradient access.
However, for non-smooth or insufficiently smooth objectives, there is a gap
between \eqref{eq:LayeredGradLB} and \eqref{eq:LayeredProxLB}.
An optimal algorithm, smoothed \ambsgd, uses the prox oracle in order
to calculate gradients of the Moreau envelope of $f(x;z)$ (cf.~Proposition 12.29
of \cite{bauschke2017convex}), and then performs \ambsgd on the smoothed
objectives. This yields a suboptimality guarantee that precisely matches
\eqref{eq:LayeredProxLB}, establishing that the lower bound from Theorem \ref{thm:proxLB} 
is tight for the layer graph, and that smoothed \ambsgd is optimal. 
An analysis of the smoothed
\ambsgd algorithm is provided in Appendix \ref{sec:smoothedambsgd}.

\natinote{Mention aggregate prox, or skip it?}

\begin{minipage}[c]{0.5\textwidth}
\subsection{Delayed updates}
\label{sec:delay}
\vspace{-1.2mm}
\end{minipage} %
\begin{minipage}[c]{0.5\textwidth}
\hfill
\begin{tikzpicture}[scale=1.0]
        \vertex(p0) at (-3,0) {};
        \vertex(p1) at (-2,0) {};
        \vertex(p2) at (-1,0) {};
        \vertex(p3) at (0,0) {};
        \vertex(p4) at (1,0) {};
        \vertex(p5) at (2,0) {};
        \vertex(p6) at (3,0) {};
    \tikzset{EdgeStyle/.style={->}}
        \draw [->] (p0) to [out=30,in=150] (p2);
        \draw [->] (p1) to [out=30,in=150] (p3);
        \draw [->] (p0) to [out=-30,in=-150] (p3);
        \draw [->] (p1) to [out=-30,in=-150] (p4);
        \draw [->] (p2) to [out=30,in=150] (p4);
        \draw [->] (p2) to [out=-30,in=-150] (p5);
        \draw [->] (p3) to [out=30,in=150] (p5);
        \draw [->] (p4) to [out=30,in=150] (p6);
        \draw [->] (p3) to [out=-30,in=-150] (p6);
    \end{tikzpicture}
\end{minipage}

We now turn to a delayed computation model that is typical in many
asynchronous parallelization and pipelined computation settings,
e.g.~when multiple processors or machines are working asynchronously,
reading iterates, taking some time to perform the oracle accesses and
computation, then communicating the results back (or updating the
iterate accordingly)
\citep{bertsekas1989parallel,nedic2001distributed,agarwal2011distributed,mcmahan2014delay,sra2016adadelay}.
This is captured by a ``delay graph'' with $T$ nodes $v_1,\ldots,v_T$
and delays $\tau_t$ for the response to the oracle query performed at $v_t$ to become available. Hence,
$\textrm{Ancestors}(v_t) = \set{v_s \mid s + \tau_s \le t}$.
Analysis is typically based on the delays being bounded, i.e. $\tau_t \leq \tau$ for all $t$. 
The depiction above corresponds to $\tau_t=2$; the case $\tau_t = 1$ corresponds to the path graph.
With constant delays $\tau_t=\tau$, the delay graph has depth
$D \le T/\tau$ 
and size $N=T$, so Theorem~\ref{thm:gradonlyLB}
gives the following lower bound when using a gradient oracle:
\begin{equation}
  \label{eq:delayedLB}
  \Omega\parens{\min\set{\frac{L}{\sqrt{T/\tau}}, \frac{H}{(T/\tau)^2}}
     + \frac{L}{\sqrt{T}}}.
\end{equation}
Delayed SGD, with updates $x_t \leftarrow x_{t-1} - \eta_t \nabla f(x_{t-\tau_t}; z)$, is a natural algorithm in this setting.
Under the bounded delay assumption
the best guarantee we are aware of for delayed update SGD is (see \citep{feyzmahdavian2016asynchronous} improving over \citep{agarwal2011distributed}) 
\begin{equation}
  \label{eq:delayedSGD}
O\left( \frac{H}{T/\tau^2} + \frac{L}{\sqrt{T}} \right).  
\end{equation}
This result is significantly worse than the lower bound
\eqref{eq:delayedLB} and quite disappointing. It does not
provide for a $1/T^2$ accelerated optimization rate, but even worse, compared to non-accelerated SGD it suffers a slowdown quadratic
in the delay, compared to the linear slowdown we would
expect.  In particular, the guarantee \eqref{eq:delayedSGD} only
allows maximum delay of $\tau = O(T^{1/4})$ in order to attain the optimal
statistical rate $\Theta(L/\sqrt{T})$, whereas the lower
bound allows a delay up to $\tau = O(T^{3/4})$.

This raises the question of whether a different algorithm
can match the lower bound \eqref{eq:delayedLB}.  The answer is
affirmative, but it requires using an
``unnatural'' algorithm, which simulates a mini-batch approach in what
seems an unnecessarily wasteful way.  We refer to this as a
``wait-and-collect'' approach: it works in $T/(2\tau)$ stages,
each stage consisting of $2\tau$ iterations (i.e.~nodes or oracle accesses).
In stage $i$, $\tau$ iterations are used to obtain
$\tau$ stochastic gradient estimates $\nabla f(x_i;z_{2 \tau i +
  j}),\,j=1,\dots,\tau$ at the same point $x_i$.  For the remaining $\tau$
iterations, we wait for all the preceding oracle computations to
become available and do not even use our allowed oracle access.  We can then
finally update the $x_{i+1}$ using the minibatch of $\tau$ gradient estimates. 
This approach is also specified formally as Algorithm
\ref{alg:wait-and-collect} in Appendix \ref{sec:waitandcollect}. Using this
approach, we can perform $T/(2\tau)$ \ambsgd updates with a
minibatch size of $\tau$, yielding a suboptimality guarantee that
precisely matches the lower bound \eqref{eq:delayedLB}.

Thus \eqref{eq:delayedLB} indeed represents the
tight complexity of the delay graph with a stochastic gradient oracle,
and the wait-and-collect approach is optimal.
However, this answer is somewhat disappointing and leaves an
intriguing open question: can a more natural, and seemingly more
efficient (no wasted oracle accesses) delayed update SGD algorithm also
match the lower bound? An answer to this question has two parts:
first, does the delayed update SGD truly suffer from a $\tau^2$ slowdown
as indicated by \eqref{eq:delayedSGD}, or does it achieve
linear degradation and a speculative guarantee of
\begin{equation}
  \label{eq:delayedMaybe}
O \left( \frac{H}{T/\tau} + \frac{L}{\sqrt{T}} \right).  
\end{equation}
Second, can delayed update SGD be accelerated to achieve the optimal
rate \eqref{eq:delayedLB}.  We note that concurrent with our work
there has been progress toward closing this gap: \citet{arjevani2018delayedsgd}
showed an improved bound matching the non-accelerated \eqref{eq:delayedMaybe}
for delayed updates (with a fixed delay) on quadratic objectives. It
still remains to generalize the result to smooth non-quadratic
objectives, handle non-constant bounded delays, and accelerate the
procedure so as to improve the rate to $(\tau/T)^2$.

\begin{minipage}[c]{0.5\textwidth}
\subsection{Intermittent communication}
\label{sec:intermittent}
\vspace{-1.2mm}
\end{minipage} %
\begin{minipage}[c]{0.5\textwidth}
\hfill
    \begin{tikzpicture}[scale=1.0]
        \vertex(p00) at (-3,-0.6) {};
        \vertex(p01) at (-2,-0.6) {};
        \vertex(p02) at (-1,-0.6) {};
        \vertex(p03) at (0,-0.6) {};
        \vertex(p04) at (1,-0.6) {};
        \vertex(p05) at (2,-0.6) {};
        \vertex(p06) at (3,-0.6) {};
        \vertex(p10) at (-3,0) {};
        \vertex(p11) at (-2,0) {};
        \vertex(p12) at (-1,0) {};
        \vertex(p13) at (0,0) {};
        \vertex(p14) at (1,0) {};
        \vertex(p15) at (2,0) {};
        \vertex(p16) at (3,0) {};
        \vertex(p20) at (-3,0.6) {};
        \vertex(p21) at (-2,0.6) {};
        \vertex(p22) at (-1,0.6) {};
        \vertex(p23) at (0,0.6) {};
        \vertex(p24) at (1,0.6) {};
        \vertex(p25) at (2,0.6) {};
        \vertex(p26) at (3,0.6) {};
    \tikzset{EdgeStyle/.style={->}}
        \Edge(p00)(p01)
        \Edge(p10)(p11)
        \Edge(p20)(p21)
        \Edge(p01)(p02)
        \Edge(p11)(p12)
        \Edge(p21)(p22)
        \Edge(p02)(p03)
        \Edge(p02)(p13)
        \Edge(p02)(p23)
        \Edge(p12)(p03)
        \Edge(p12)(p13)
        \Edge(p12)(p23)
        \Edge(p22)(p03)
        \Edge(p22)(p13)
        \Edge(p22)(p23)
        \Edge(p03)(p04)
        \Edge(p13)(p14)
        \Edge(p23)(p24)
        \Edge(p04)(p05)
        \Edge(p14)(p15)
        \Edge(p24)(p25)
        \Edge(p05)(p06)
        \Edge(p05)(p16)
        \Edge(p05)(p26)
        \Edge(p15)(p06)
        \Edge(p15)(p16)
        \Edge(p15)(p26)
        \Edge(p25)(p06)
        \Edge(p25)(p16)
        \Edge(p25)(p26);
    \end{tikzpicture}
\end{minipage}

We now turn to a parallel computation model which is relevant
especially when parallelizing across disparate machines: in each of $T$
iterations, there are $M$ machines that, instead of just a single oracle
access, perform $K$ sequential oracle accesses before
broadcasting to all other machines synchronously.  This communication
pattern is relevant in the realistic scenario where local computation is plentiful 
relative to communication costs (i.e.~$K$ is large). 
This may be the case with fast processors distributed across
different machines, or in the setting of federated learning, 
where mobile devices collaborate to
train a shared model while keeping their respective training datasets
local \cite{mcmahan17fedavg}.

This is captured by a graph consisting of $M$ parallel chains of
length $TK$, with cross connections between the chains every $K$
nodes.  Indexing the nodes as $v_{t,m,k}$, the nodes
$v_{t,m,1}\rightarrow \cdots \rightarrow v_{t,m,K}$ form a chain, and
$v_{t,m,K}$ is connected to $v_{t+1,m',1}$ for all $m'=1..M$.  This
graph generalizes the layer graph by allowing $K$ sequential oracle
queries between each complete synchronization; $K=1$ recovers the
layer graph, and 
the depiction above corresponds to $\K=\machines=3$.
We refer to the computation between each synchronization step as a
(communication) round.

\removed{In
particular, algorithms where a set of devices synchronize by computing
an average update are particularly attractive, as such algorithms can
be extended to provide stronger privacy guarantees, e.g. via
cryptographic secure aggregation protocols \cite{bonawitz17secagg} or
differential privacy\cite{mcmahan18dplm}.}

The depth of this graph is $D=TK$ and the size is $N=TKM$.
Focusing on the stochastic gradient oracle (the situation is similar
for the prox oracle, except with the potential of smoothing a
non-smooth objective, as discussed in Section \ref{sec:layered}),
Theorem~\ref{thm:gradonlyLB} yields the lower bound:
\begin{equation}
  \label{eq:intLB}
\Omega\parens{\min\set{\frac{L}{\sqrt{T \K}},\ \frac{H}{T^2\K^2}} + \frac{L}{\sqrt{T \K \machines}}}.
\removed{%
\qquad \text{and} \quad
\Omega\parens{\min\set{\frac{L}{T \K},\ \frac{H}{T^2 \K^2}} + \frac{L}{\sqrt{T \K \machines}}}}
\end{equation}
\removed{for the gradient and prox oracles respectively.}

A natural algorithm for this graph is parallel SGD, where we run an SGD chain on
each machine and average iterates during communication rounds, e.g. \cite{mcmahan17fedavg}.  The updates are then given by: 
\begin{equation}
  \label{eq:parSGD}
  \begin{aligned}
  x_{t,m,0} &= \frac{1}{M} \sum_{m'} x_{t,m',K} \\
  x_{t,m,k} &= x_{t,m,k-1} - \eta_t \nabla f(x_{t,m,k-1};z_{t,m,k}),\, k=1,\dots,K
\end{aligned}
\end{equation}
(note that $x_{t,m,0}$ does not correspond to any node in the graph, and is included for convenience of presentation).
Unfortunately, we are not aware of any satisfying analysis of such a
parallel SGD approach. 
Instead, we consider two other algorithms in an attempt to match the lower bound \eqref{eq:intLB}.
First, we can combine all $\K\machines$
oracle accesses between communication rounds in order to form a single mini-batch,
giving up on the possibility of sequential computation along the ``local''
$\K$ node sub-paths.  Using all $KM$ nodes to obtain stochastic
gradient estimates at the same point, we can perform $T$ iterations of \ambsgd 
with a mini-batch size of $KM$, yielding an upper bound
of
\begin{equation}
  \label{eq:intMB}
 O\parens{
    \frac{H}{T^2}
    + \frac{L}{\sqrt{T \K \machines}}}.
\end{equation}
This is a reasonable and common approach, and it is optimal (up to constant factors) when
$KM=O(\frac{L^2}{H^2} T^3)$ so that the statistical term is limiting.  However,  comparing
\eqref{eq:intMB} to the lower bound \eqref{eq:intLB} we see a gap by a
factor of $K^2$ in the optimization term, indicating the possibility
for significant gains when $K$ is large (i.e.~when we can process a
large number of examples on each machine at each round). Improving the optimization term by
this $K^2$ factor would allow statistical
optimality as long as $M=O(T^3K^3)$----this is a very significant
difference. In many scenarios we would expect a modest number of
machines, but the amount of data on each machine could
easily be much more than the number of communication rounds,
especially if communication is across a wide area network.

In fact, when $K$ is large, a different approach is preferable: we can ignore all but a single chain and simply execute $\K T$ iterations of sequential SGD, offering an upper bound of
\begin{equation}
  \label{eq:parSingle}
O\parens{\frac{L}{\sqrt{T \K}}}.
\end{equation}
Although this approach seems extremely wasteful, it actually yields a
better guarantee than \eqref{eq:intMB} when
$K \geq \Omega(T^3L^2/H)$. This is a realistic regime, e.g.~in federated learning when computation is distributed across devices, communication is limited and sporadic and so only a
relatively small number of rounds $T$ are possible, but each device already possesses a large amount of data. Furthermore, for non-smooth functions, \eqref{eq:parSingle} matches the lower bound \eqref{eq:intLB}.

Our upper bound on the complexity is therefore obtained by selecting
either \ambsgd or single-machine
sequential SGD, yielding a combined upper bound of
\begin{equation}
  \label{eq:parUB}
  O\parens{\min\set{ \frac{L}{\sqrt{T \K}},\ \frac{H}{T^2}}
    + \frac{L}{\sqrt{T \K \machines}}. }
\end{equation}
For smooth functions, there is still a significant gap between this upper bound and the
lower bound \eqref{eq:intLB}. Furthermore, this upper bound is not
achieved by a single algorithm, but rather a combination of two
separate algorithms, covering two different regimes. This
raises the question of whether there is a single, natural algorithm,
perhaps an accelerated variant of the parallel SGD updates
\eqref{eq:parSGD}, that at the very least matches \eqref{eq:parUB}, 
and preferably also improves over them in the intermediate regime 
or even matches the lower bound \eqref{eq:intLB}.



\paragraph{Active querying and SVRG} All methods discussed so far used
fully stochastic oracles, requesting a gradient (or prox
computation) with respect to an independently and randomly drawn
$z\sim\Pcal$.  We now turn to methods that also make active queries,
i.e.~draw samples from $\Pcal$ and then repeatedly query the oracle, at
different points $x$, but on the same samples $z$.  Recall that all of
our lower bounds are valid also in this setting.  

With an active query gradient oracle, we can implement SVRG
\cite{johnson13svrg,lee2017distributed} on an intermittent
communication graph.  More specifically, for an appropriate choice of $n$ and $\lambda$, 
we apply SVRG to the regularized empirical objective $\hat{F}_{\lambda}(x) = \frac{1}{n}\sum_{i=1}^n f(x;z_i) + \frac{\lambda}{2}\norm{x}^2$

\begin{algorithm}
  \caption{SVRG}
  \label{alg:svrg}
  \begin{algorithmic}
    \STATE Parameters: $n, S, I$,$\quad$ Sample $z_1,\dots,z_n \sim \Pcal$,$\quad$ Initialize $x_0 = 0$
    \FOR{$s=1,2,\dots,S=\left\lfloor T/ \parens{\left\lceil\frac{n}{KM}\right\rceil + \left\lceil\frac{I}{K}\right\rceil}\right\rfloor$}
    \STATE $\tilde{x} = x_{s-1},\quad x_s^0 = \tilde{x}$
    \STATE $\tilde{g} = \nabla \hat{F}_{\lambda}(\tilde{x}) = \frac{1}{n}\sum_{i=1}^n \nabla f(\tilde{x};z_i) + \lambda \tilde{x}$ \hfill $(*)$
    \FOR{$i = 1,2,\dots,I=\frac{H}{\lambda}$}
    \STATE Sample $j \sim \textrm{Uniform}\set{1,\dots,n}$
    \STATE $x_s^i = x_s^{i-1} - \eta \parens{\parens{\nabla f(x_s^{i-1};z_j) + \lambda x_s^{i-1}} - \parens{\nabla f(\tilde{x};z_j) + \lambda \tilde{x}} + \tilde{g}}$ \hfill $(**)$
    \ENDFOR
    \STATE $x_{s} = x_s^i$ for $i \sim \textrm{Uniform}\set{1,\dots,I}$
    \ENDFOR
    \STATE \textbf{Return} $x_S$
  \end{algorithmic}
\end{algorithm}

To do so, we first pick a sample $\{z_1, \dots z_n\}$ (without actually
querying the oracle). As indicated by Algorithm \ref{alg:svrg}, we then alternate between computing full gradients on $\{z_1, \dots z_n\}$ in parallel $(*)$, and sequential variance-reduced stochastic gradient updates in between $(**)$. The full gradient $\tilde{g}$ is computed using $n$ active queries to the gradient oracle. Since all of these oracle accesses are made at the same point $\tilde{x}$, this can be fully parallelized across the $M$ parallel chains of length $K$ thus requiring $n/KM$ rounds. The sequential variance-reduced stochastic gradient updates \emph{cannot} be parallelized in this way, and must be performed using queries to the gradient oracle in just one of the $M$ available parallel chains, requiring $I/K$ rounds of synchronization. Consequently, each outer iteration of SVRG requires $\left\lceil\frac{n}{KM}\right\rceil + \left\lceil\frac{I}{K}\right\rceil$ rounds. We analyze this method using $\lambda = \Theta\parens{\frac{L}{\sqrt{n}}}$, $I = \Theta\parens{\frac{H}{\lambda}} = \Theta\parens{\frac{H\sqrt{n}}{L}}$, and $n = \min\set{\Theta\parens{\frac{K^2T^2L^2}{H^2\log^2\parens{MKT/L}}}, \Theta\parens{\frac{MKT}{\log\parens{MKT/L}}}}$. Using the analysis of \citet{johnson13svrg}, SVRG guarantees that, with an appropriate stepsize, we have $\hat{F}_\lambda(x_S) - \min_{x}\hat{F}_\lambda(x) \leq 2^{-S}$; the value of $x_S$ on the empirical objective also generalizes to the population, so $\E{f(x_S;z)} - \min_{x}\E{f(x;z)} \leq 2^{-S} + O\parens{\frac{L}{\sqrt{n}}}$ (see \cite{shalev2008svm}). With our choice of parameters, this implies upper bound (see Appendix \ref{sec:analysis_svrg})
\begin{equation}
  \label{eq:parSVRG}
  O\parens{\parens{ \frac{H}{T \K} + \frac{L}{\sqrt{T K \machines} }}
    \log\parens{ \frac{T \K \machines}{ L}}}.  
\end{equation}

These guarantees improve over sequential SGD
\eqref{eq:parSGD} as soon as  $\machines > \log^2(T \K \machines / L)$
and $K>H^2/L^2$, i.e.~$L/\sqrt{TK} < L^2/H$.  This is a very wide
regime: we require only a moderate number of machines, and the second condition 
will typically hold for a smooth loss. Intuitively, SVRG does 
roughly the same number (up to a factor of two) of
sequential updates as in the sequential SGD approach but it uses
better, variance reduced updates. The price
we pay is in the smaller total sample size since we keep calling the
oracle on the same samples. Nevertheless, since SVRG only needs to calculate the
``batch'' gradient a logarithmic number of times, this incurs only an additional
logarithmic factor.  

Comparing \eqref{eq:intMB} and \eqref{eq:parSVRG}, we see that SVRG
also improves over \ambsgd as soon as $\K > T \log(T \K \machines/L)$,
that is if the number of points we are processing on each machine each
round is slightly more then the total number of rounds,
which is also a realistic scenario.

To summarize, the best known upper bound for optimizing with
intermittent communication using a pure stochastic oracle is \eqref{eq:parUB},
 which combines two different algorithms. However,
with active oracle accesses, SVRG is also possible and
the upper bound becomes:
\begin{equation}
  O\parens{\min\set{ \frac{L}{\sqrt{T \K}},\ \parens{ \frac{H}{T \K} + \frac{L}{\sqrt{T K \machines} }}
    \log\parens{ \frac{T \K \machines}{ L}},\    \frac{H}{T^2}
    + \frac{L}{\sqrt{T \K \machines}}} }
\end{equation}


\section{Summary}

Our main contributions in this paper are: (1) presenting a precise formal
oracle framework for studying parallel stochastic
optimization; (2) establishing tight oracle lower bounds in this
framework that can then be easily applied to particular instances of parallel optimization;
and (3) using the framework to study specific settings,
obtaining optimality guarantees, understanding where additional
assumptions would be needed to break barriers, and, perhaps most
importantly, identifying gaps in our understanding that highlight
possibilities for algorithmic improvement.
Specifically,
\begin{itemize}
  \item For non-smooth objectives and a stochastic prox
  oracle, smoothing and acceleration can improve performance in the 
  layer graph setting. It is not clear if there is a more direct 
  algorithm with the same optimal performance, e.g.~averaging the answers from the prox oracle.
  \item In the delay graph setting, delayed update SGD's guarantee is not
  optimal.  We suggest an
  alternative optimal algorithm, but it would be interesting
  and beneficial to understand the true behavior of delayed update
  SGD and to improve it as necessary to attain optimality.
  \item With intermittent communication, we show how different methods
  are better in different regimes, but even combining these methods
  does not match our lower bound.  This raises the question of whether our lower bound is
  achievable.  Are current methods optimal?  Is the true optimal
  complexity somewhere in between?  Even finding a single
  method that matches the current best performance in all
  regimes would be a significant advance here.
  \item With intermittent communication, active queries 
  allow us to obtain better performance in a certain
  regime.  Can we match this performance using pure
  stochastic queries or is there a real gap between active and pure
  stochastic queries?
\end{itemize}

The investigation into optimizing over $\FLHB$ in our framework
indicates that there is no advantage to the prox oracle for
optimizing (sufficiently) smooth functions. This raises the question
of what additional assumptions might allow us to leverage the prox
oracle, which is intuitively much stronger as it allows global
access to $f(\cdot;z)$. 
One option is to assume a bound on the variance of the
stochastic oracle i.e.~$\mathbb{E}_z [ \norm{\nabla f(x;z) - \nabla
  F(x)}^2 ] \leq \sigma^2$ which captures the notion that the
functions $f(\cdot;z)$ are somehow related and not arbitrarily
different.  In particular, if each stochastic oracle access, in each
node, is based on a sample of $b$ data points (thus, a prox
operation optimizes a sub-problem of size $b$), we have that
$\sigma^2 \leq L^2/b$.  Initial investigation into the complexity of
optimizing over the restricted class $\FLHBs$ (where we also require
the above variance bound), reveals a significant theoretical
advantage for the prox oracle over the gradient oracle, even for
smooth functions.  This is an example of how formalizing the
optimization problem gives insight into additional assumptions, in
this case low variance, that are necessary for realizing the
benefits of a stronger oracle.

\subsection*{Acknowledgements}
We would like to thank Ohad Shamir for helpful discussions. This work was partially funded by NSF-BSF award 1718970 (``Convex and
Non-Convex Distributed Learning'') and a Google Research Award.  BW is
supported by the NSF Graduate Research Fellowship under award
1754881. AS was supported by NSF awards  IIS-1447700 and AF-1763786, as well as a
Sloan Foundation research award. 

\bibliography{graphs_nips18}
\bibliographystyle{plainnat}

\clearpage
\appendix

\section{Main lower bound lemma}\label{sec:vectorguessing}
This analysis closely follows that of previous work, specifically the proof of Theorem 1 in \cite{woodworth2017lower} and the proof of Lemma 4 in \cite{carmon2017lower}. There are slight differences in the problem setup between this work and that of previous papers, thus we include the following analysis for completeness and to ensure that all of our results can be verified. We do not claim any significant technical novelty within this section.

Let $V = \set{v_1,\dots,v_k}$ be a uniformly random orthonormal set of vectors in $\R^m$. All of the probabilities referred to in Appendix \ref{sec:vectorguessing} will be over the randomness in the selection of $V$.
Let $X = \set{x_1,x_2,\dots,x_N}$ be a set of vectors in $\R^m$ where $\norm{x_i} \leq 1$ for all $i \leq N$. Let these vectors be organized into disjoint subsets $X_1 \cup X_2 \cup \dots \cup X_k = X$. 
Furthermore, suppose that for each $t \leq k$, the set $X_t$ is a deterministic function $X_t = X_t(X_{<t},V)$, so it can also be expressed as $X_t = X_t(V)$.

Let $S_t = X_{\leq t} \cup V_{\leq t}$, let $P_t$ be the projection operator onto the span of $S_t$ and let $P^\perp_t$ be the projection onto the orthogonal complement of the span of $S_t$.
As in \cite{woodworth2017lower,carmon2017lower}, define
\begin{equation}\label{eq:defgt}
G_t = G_t(V) = \left\llbracket \forall x \in X_t\ \forall j \geq t\  \abs{\inner{\normalized{P_{t-1}^\perp x}}{v_j}} \leq \alpha \right\rrbracket
\end{equation}
Finally, suppose that for each $t$, $X_t$ is of the form:
\begin{equation}\label{eq:xtdecomposes}
X_t\parens{V} = X_t\parens{V_{<t}\indicator{G_{<t}} + V\indicator{\lnot G_{<t}}}
\end{equation}
i.e. conditioned on the event $G_{<t}$, it is a deterministic function of $V_{<t}$ only (and not $v_t,...,v_k$). We say that $\Prob{G_{<1}} = 1$, so $X_1$ is always independent of $V$.

First, we connect the events $G_t$ to a more immediately useful condition
\begin{restatable}{lemma}{Gksmallips}[cf. Lemma 9 \cite{carmon2017lower}, Lemma 1 \cite{woodworth2017lower}] \label{lem:Gksmallips}
For any $c$, $k$, $N$, $V$, and $\set{X_k}_{t=1}^k$, let 
$\alpha = \min\set{\frac{1}{4N},\ \frac{c}{2\parens{1 + \sqrt{2N}}}}$ then for each $t \leq k$
\[
G_{\leq t} \implies G'_{\leq t} := \left\llbracket \forall r \leq t,\ \forall x \in X_r,\ \forall j\geq t\ \abs{\inner{x}{v_j}} \leq \frac{c}{2} \right\rrbracket
\]
\end{restatable}
The proof of Lemma \ref{lem:Gksmallips} involves straightforward linear algebra, and we defer it to Appendix \ref{sec:vectorguessingdeferred}. By Lemma \ref{lem:Gksmallips}, $G_{<t} \subseteq G'_{<t}$, therefore the property \eqref{eq:xtdecomposes} is implied by
\begin{equation}\label{eq:xtdecomposesuseful}
X_t\parens{V} = X_t\parens{V_{<t}\indicator{G'_{<t}} + V\indicator{\lnot G'_{<t}}}
\end{equation}


Now, we state the main result which allows us to prove our lower bounds:
\begin{restatable}{lemma}{keylowerboundlemma}[cf. Lemma 4 \cite{carmon2017lower}, Lemma 4 \cite{woodworth2017lower}]\label{lem:keylowerboundlemma}
For any $k \geq 1$, $N \geq 1$, $c \in (0,1)$, and dimension 
\[
m \geq k+N+\max\set{32N^2,\ \frac{8(1+\sqrt{2N})^2}{c^2}}\log\parens{2k^2N}
\]
if the sets $X_1,\dots,X_k$ satisfy the property \eqref{eq:xtdecomposesuseful} then
\[
\Prob{\forall t \leq k\ \forall x \in X_t\ \forall j \geq t\ \ \abs{\inner{x}{v_j}} \leq \frac{c}{2}} \geq \frac{1}{2}
\]
\end{restatable}
The proof of Lemma \ref{lem:keylowerboundlemma} relies upon the following, whose proof we defer to Appendix \ref{sec:vectorguessingdeferred}.
\begin{restatable}{lemma}{symmetricdensity}[cf. Lemma 11 \cite{carmon2017lower}, Lemma 3 \cite{woodworth2017lower}] \label{lem:symmetricconditionaldensity}
Let $R$ be any rotation operator, $R^\top R = I$, that preserves $S_{t-1}$, that is $Rw = R^\top w = w$ for any $w \in \textrm{Span}\parens{S_{t-1}}$. Then the following conditional densities are equal
\begin{equation*}
p\parens{V_{\geq t}\ \middle|\ G_{<t}, V_{<t}} = p\parens{RV_{\geq t}\ \middle|\ G_{<t}, V_{<t}}
\end{equation*}
\end{restatable}

\begin{proof}[Proof of Lemma \ref{lem:keylowerboundlemma}]
This closely follows the proof of Lemma 4 \cite{carmon2017lower} and Lemma 4 \cite{woodworth2017lower}, with small modifications to account for the different setting.

Set $\alpha = \min\set{\frac{1}{4N},\,\frac{c}{2\parens{1+\sqrt{2N}}}}$. Then by Lemma \ref{lem:Gksmallips}, since $X_1,\dots,X_k$ satisfy the property \eqref{eq:xtdecomposesuseful}
\begin{equation}
\Prob{\forall t \leq k\ \forall x \in X_t\ \forall j \geq t\ \ \abs{\inner{x}{v_j}} \leq \frac{c}{2}} 
\geq
\Prob{G_{\leq k}} = \prod_{t \leq k}\ProbGiven{G_t}{G_{<t}}
\end{equation}
Focus on a single term in this product,
\begin{equation}
\mathbb{P}\left[ G_t\ \middle|\ G_{<t} \right] = \mathbb{E}_{V_{<t}}\left[\mathbb{P}\left[ G_t\ \middle|\ G_{<t}, V_{<t} \right]\right] \geq \inf_{V_{<t}} \mathbb{P}\left[ G_t\ \middle|\ G_{<t}, V_{<t}\right]
\end{equation}
For any particular $V_{<t}$,
\begin{align}
&\mathbb{P}\left[ G_t\ \middle|\ G_{<t}, V_{<t} \right]
= \ProbGiven{\forall x \in X_t\ \forall j \geq t\ \abs{\inner{\normalized{P_{t-1}^{\perp}x}}{v_j}} \leq \alpha}{G_{<t}, V_{<t}} \\
&\geq 1- \sum_{x \in X_t(V_{<t})} \sum_{j=t}^k \ProbGiven{\abs{\inner{\normalized{P_{t-1}^{\perp}x}}{v_j}} > \alpha}{G_{<t}, V_{<t}} \\
&\geq 1- \sum_{x \in X_t(V_{<t})} \sum_{j=t}^k \ProbGiven{\abs{\inner{\normalized{P_{t-1}^{\perp}x}}{\normalized{P_{t-1}^\perp v_j}}} > \alpha}{G_{<t}, V_{<t}}
\end{align}
Conditioned on $G_{<t}$ and $V_{<t}$, the set $X_t = X_t(V_{<t})$ is fixed, as is the set $S_{t-1}$ and therefore $P_{t-1}^{\perp}$, so the first term in the inner product is a fixed unit vector. By Lemma \ref{lem:symmetricconditionaldensity}, the conditional density of $v_j\ |\ G_{<t},V_{<t}$ is spherically symmetric within the span onto which $P_{t-1}^\perp$ projects. Therefore, $\normalized{P_{t-1}^\perp v_j}$ is distributed uniformly on the unit sphere in $\textrm{Span}\parens{S_{t-1}}^\perp$, which has dimension at least $m' := m - (t-1) - \sum_{r=1}^{t-1}\abs{X_{r}} \geq m - k + 1 - N$. 

The probability of a fixed vector and a uniform random vector on the unit sphere in $\R^{m'}$ having inner product more than $\alpha$ is proportional to the surface area of the ``end caps" of the sphere lying above and below circles of radius $\sqrt{1 - \alpha^2}$, which is strictly smaller than the surface area of a full sphere of radius $\sqrt{1 - \alpha^2}$. Therefore, for a given $x,v_j$
\begin{align} 
\ProbGiven{\abs{\inner{\normalized{P_{t-1}^{\perp}x}}{\normalized{P_{t-1}^\perp v_j}}} > \alpha}{G_{<t}, V_{<t}} 
&< \frac{\textrm{SurfaceArea}_{m'}(\sqrt{1 - \alpha^2})}{\textrm{SurfaceArea}_{m'}(1)} \\
&= \parens{\sqrt{1 - \alpha^2}}^{m'-1} \\
&\leq \expo{-\frac{(m'-1)\alpha^2}{2}} 
\end{align}
where we used that $1-x \leq \exp(-x)$. Finally, this holds for each $t$, $x \in X_t$, and $j \geq t$, so
\begin{align}
\Prob{G_{\leq k}}
&\geq \prod_{t \leq k}\inf_{V_{<t}} \mathbb{P}\left[ G_t\ \middle|\ G_{<t}, V_{<t} \right] \\
&\geq \parens{1 - kN \expo{-\frac{\parens{m - k - N}\alpha^2}{2}}}^k \\
&\geq 1 - k^2N\expo{-\frac{\alpha^2}{2}\max\set{32N^2,\ \frac{8(1+\sqrt{2N})^2}{c^2}}\log\parens{\frac{1}{2k^2N}}} \label{eq:weusedm} \\
&= \frac{1}{2}\label{eq:weusedalpha}
\end{align}
Where we used that $m \geq k+N+\max\set{32N^2,\ \frac{8(1+\sqrt{2N})^2}{c^2}}\log\parens{2k^2N}$ for \eqref{eq:weusedm}. For \eqref{eq:weusedalpha}, recall that we chose $\alpha = \min\set{\frac{1}{4N},\ \frac{c}{2(1+\sqrt{2N})}}$ so $\max\set{32N^2,\ \frac{8(1+\sqrt{2N})^2}{c^2}} = \frac{2}{\alpha^2}$.
\end{proof}

\subsection{Proof of Lemmas \ref{lem:Gksmallips} and \ref{lem:symmetricconditionaldensity}}\label{sec:vectorguessingdeferred}
\Gksmallips*
\begin{proof}
This closely follows the proof of Lemma 9 \cite{carmon2017lower}, with slight modification to account for the different problem setup.

For $t \leq k$ assume $G_{\leq t}$. For any $x \in X_t$ and $j \geq t$
\begin{align}
\abs{\inner{x}{v_j}} 
&\leq \norm{x}\abs{\inner{\normalized{x}}{P_{t-1}v_j}} + \norm{x}\abs{\inner{\normalized{x}}{P_{t-1}^{\perp}v_j}} \\
&\leq \norm{P_{t-1}v_j} + \abs{\inner{\frac{P_{t-1}^{\perp}x}{\norm{x}}}{v_j}} \\
&\leq \norm{P_{t-1}v_j}  + \abs{\inner{\normalized{P_{t-1}^{\perp}x}}{v_j}} \\
&\leq \norm{P_{t-1}v_j}  + \alpha \label{eq:ipxvj}
\end{align}
First, we decomposed $v_j$ into its $S_{t-1}$ and $S^\perp_{t-1}$ components and applied the triangle inequality. Next we used that $\norm{x} \leq 1$ and that the orthogonal projection operator $P_{t-1}^\perp$ is self-adjoint. Finally, we used that the projection operator is non-expansive and the definition of $G_t$.

Next, we prove by induction on $t$ that for all $t \leq k$ and $j \geq t$, the event $G_{\leq t}$ implies that $\norm{P_{t-1}v_j}^2 \leq 2\alpha^2\sum_{r=1}^{t-1}\abs{X_r}$.
As a base case ($t=1$), observe that, trivially, $\norm{P_{t-1}v_j}^2 = \norm{0v_j}^2 = 0$.
For the inductive step, fix any $t \leq k$ and $j \geq t$ and suppose that $G_{\leq t'} \implies \norm{P_{t'-1}v_j'}^2 \leq 2\alpha^2\sum_{r=1}^{t'-1}\abs{X_r}$ for all $t' < t$ and $j' \geq t'$. Let $\hat{P}_t$ project onto $\spn{S_t \cup X_{t+1}}$ (this includes $X_{t+1}$ in contrast with $P_t$) and let $\hat{P}_t^\perp$ project onto the orthogonal subspace. Since $\spn{X_1\cup X_2\cup\dots\cup X_{t-1}\cup V_{\leq t-1}} = S_{t-1}$,
\begin{equation} \label{eq:gramschmidtbasis}
\set{\normalized{P_{r-1}^\perp x}\,:\, r \leq t-1,\ x \in X_r} \cup \set{\normalized{\hat{P}_{r-1}^\perp v_r}\,:\, r \leq t-1}
\end{equation} 
is a (potentially over-complete) basis for $S_{t-1}$. Using the triangle inequality and $G_{< t}$, we can therefore expand
\begin{align}
\norm{P_{t-1} v_j}^2 
&= \sum_{r=1}^{t-1}\sum_{x\in X_r} \inner{\normalized{P_{r-1}^\perp x}}{v_j}^2
       + \sum_{r=1}^{t-1} \inner{\normalized{\hat{P}_{r-1}^\perp v_r}}{v_j}^2 \\
&\leq \alpha^2\sum_{r=1}^{t-1}\abs{X_r} + \sum_{r=1}^{t-1} \frac{1}{\norm{\hat{P}_{r-1}^\perp v_r}^2}\inner{\hat{P}_{r-1}^\perp v_r}{v_j}^2 \label{eq:needtoboundsecondterm}
\end{align}
We must now bound the second term of \eqref{eq:needtoboundsecondterm}. Focusing on the inner product in the numerator for one particular $r < t$:
\begin{align} 
\abs{\inner{\hat{P}_{r-1}^\perp v_r}{v_j}} 
&= \abs{\inner{v_r}{v_j} - \inner{\hat{P}_{r-1}v_r}{v_j}} \\
&= \abs{\inner{\hat{P}_{r-1}v_r}{v_j}} \\
&\leq \abs{\inner{P_{r-1}v_r}{v_j}} + \sum_{x\in X_r}\abs{\inner{\normalized{P_{r-1}^\perp x}}{v_r}\inner{\normalized{P_{r-1}^\perp x}}{v_j}} \\
&\leq \norm{P_{r-1}v_r}\norm{P_{r-1}v_j} + \abs{X_r}\alpha^2 \label{eq:cauchyschwarzandgt}\\
&\leq 2\alpha^2\sum_{i=1}^{r-1}\abs{X_i} + \abs{X_r}\alpha^2  \\
&\leq \frac{\alpha}{2} \label{eq:boundinnerproduct}
\end{align}
First, we used that $\hat{P}_{r-1}^\perp = I - \hat{P}_{r-1}$, then that $v_r \perp v_j$. Next, we applied the definition of $\hat{P}_{r-1}$ and the triangle inequality. To get \eqref{eq:cauchyschwarzandgt} we use the Cauchy-Schwarz inequality on the first term, and the definition of $G_r$ for the second. Finally, we use the inductive hypothesis and that $\alpha \leq \frac{1}{4N}$.

We have now upper bounded the inner products in the second term of \eqref{eq:needtoboundsecondterm}, and it remains to lower bound the norm in the denominator. We can rewrite
\begin{align}
\norm{\hat{P}_{r-1}^\perp v_r}^2 
&= \inner{\hat{P}_{r-1}^\perp v_r}{v_r} \\
&= \inner{v_r}{v_r} - \inner{\hat{P}_{r-1}v_r}{v_r} \\
&\geq 1 - \inner{P_{r-1}v_r}{v_r} - \sum_{x\in X_r} \inner{\normalized{P_{r-1}^\perp x}}{v_r}^2 \\
&\geq 1 - \norm{P_{r-1}v_r}^2 - \abs{X_r} \alpha^2 \\
&\geq 1 - 2\alpha^2\sum_{i=1}^{r-1}\abs{X_i} - \abs{X_r} \alpha^2 \\
&\geq \frac{1}{2} \label{eq:denominatornorm}
\end{align}
Here we again used $\hat{P}_{r-1}^\perp = I - \hat{P}_{r-1}$ followed by an (over)expansion of $\hat{P}_{r-1}$. The remaining steps follow from the inductive hypothesis and fact that $\alpha \leq \frac{1}{4N}$. Combining \eqref{eq:denominatornorm} with \eqref{eq:boundinnerproduct} and returning to \eqref{eq:needtoboundsecondterm}, we have that
\begin{align}
\norm{P_{t-1} v_j}^2 
&\leq \alpha^2\sum_{r=1}^{t-1}\abs{X_r} + \sum_{r=1}^{t-1} \frac{1}{\norm{\hat{P}_{r-1}^\perp v_r}^2}\inner{\hat{P}_{r-1}^\perp v_r}{v_j}^2 \\
&\leq \alpha^2\sum_{r=1}^{t-1}\abs{X_r} + \sum_{r=1}^{t-1} \alpha^2 \\
&\leq 2\alpha^2\sum_{r=1}^{t-1}\abs{X_r}
\end{align}
Therefore, for each  $t \leq k$ and $j \geq t$ an upper bound $\norm{P_{t-1} v_j}^2 \leq 2\alpha^2\sum_{r=1}^{t-1}\abs{X_r}$. Returning now to \eqref{eq:ipxvj}, we have that for any $t \leq k$, $x \in X_t$, and $j \geq t$ the event $G_{\leq t}$ implies
\begin{align}
\abs{\inner{x}{v_j}} 
&\leq \norm{P_{t-1}v_j} + \alpha \\
&\leq \alpha\parens{1 + \sqrt{2\sum_{r=1}^{t-1}\abs{X_r}}} \\
&\leq \frac{c}{2}
\end{align}
where we used that $\alpha \leq \frac{c}{2\parens{1+\sqrt{2N}}}$
\end{proof}

\symmetricdensity*
\begin{proof}
This closely follows the proof of Lemma 11 \cite{carmon2017lower}.

First, we apply Bayes' rule to each density and use the fact that $RV_{<t} = V_{<t}$:
\begin{align}
p\parens{V_{\geq t}\ \middle|\ G_{<t}, V_{<t}} &=
\frac{\ProbGiven{G_{<t}}{V}\ p(V)}{\ProbGiven{G_{<t}}{V_{<t}}\ p(V_{<t})} \label{eq:guessinglemmaunrotated}\\
p\parens{RV_{\geq t}\ \middle|\ G_{<t}, V_{<t}} &=
\frac{\ProbGiven{G_{<t}}{RV}\ p(RV)}{\ProbGiven{G_{<t}}{V_{<t}}\ p(V_{<t})} \label{eq:guessinglemmarotated}
\end{align}
Since $V$ has a spherically symmetric marginal distribution, $p(V) = p(RV)$. Therefore, it only remains to show that $\ProbGiven{G_{<t}}{V} = \ProbGiven{G_{<t}}{RV}$. The event $G_{<t}$ is determined by $V$ or by $RV$, thus both probabilities are either $0$ or $1$, so it suffices to show $\ProbGiven{G_{<t}}{V} = 1 \iff \ProbGiven{G_{<t}}{RV} = 1$.

Assume first $\ProbGiven{G_{<t}}{V} = 1$. Then for each $r < t$, $x \in X_r$, and $j \geq r$ $\abs{\inner{\normalized{P_{r-1}^\perp x}}{v_j}} \leq \alpha$, and each set $X_r$ is a deterministic function of $V_{<r}$. Also, observe that for any $x \in X_r$ and $j \geq r$, 
\begin{equation}
\abs{\inner{\normalized{P_{r-1}^\perp x}}{Rv_j}} = \abs{\inner{\frac{R^\top P_{r-1}^\perp x}{\norm{P_{r-1}^\perp x}}}{v_j}} = \abs{\inner{\normalized{P_{r-1}^\perp x}}{v_j}} \leq \alpha
\end{equation}
where we used that $P_{r-1}^\perp x \in \textrm{Span}(S_{r}) \subseteq \textrm{Span}(S_{t-1})$ so $R^\top P_{r-1}^\perp x = P_{r-1}^\perp x$. Therefore, it suffices to show that the sequence $X_1(RV),...,X_t(RV) = X_1(V),...,X_t(V)$ when $\ProbGiven{G_{<t}}{V} = 1$. We prove this by induction.

For the base case, by definition $X_1(RV) = X_1 = X_1(V)$. For the inductive step, suppose now that $X_{r'}(RV) = X_{r'}(V)$ for each $r' < r$. This, plus the fact that $\ProbGiven{G_{<t}}{V} = 1 \implies \ProbGiven{G_{<r}}{V} = 1$ together imply that $\ProbGiven{G_{<r}}{RV} = 1$. Thus, $X_r(RV) = X_r(RV_{<r}) = X_r(V_{<r})$. Therefore, we conclude that $\ProbGiven{G_{<t}}{V} = 1 \implies \ProbGiven{G_{<t}}{RV} = 1$, the reverse implication can be proven with a similar argument.
\end{proof}

\section{Proof of Theorem \ref{thm:gradonlyLB}}\label{sec:gradonlyappendix}
\gradonlyLB*
\begin{proof}
Assume for now that $B = 1$, the lower bound can be established for other values of $B$ by scaling inputs to our construction.
Let
\begin{equation}
\ell = \min\set{L,\ \frac{H}{10(D+1)^{1.5}}} \qquad\qquad \eta = 10(D+1)^{1.5}\ell
\end{equation}
and consider the following $\ell$-Lipschitz function:
\begin{equation}
\tilde{f}(x) = \max_{1\leq r \leq D+1} \ell v_r^\top x - \frac{5\ell^2(r-1)}{\eta}
\end{equation}
where the vectors $v_1,\dots,v_{D+1}$ are an orthonormal set drawn uniformly at random from the unit sphere in $\R^m$. We use the $\eta$-Moreau envelope \cite{bauschke2017convex} of this function in order to prove our lower bound:
\begin{align}
f(x) 
&= \inf_{y}\set{\tilde{f}(y) + \frac{\eta}{2}\norm{y-x}^2} 
\end{align}
The random draw of $V$ defines a distribution over functions $f$. We will lower bound the expected suboptimality of any \emph{deterministic} optimization algorithm's output and apply Yao's minimax principle at the end of the proof.

This function has the following properties:
\begin{restatable}{lemma}{moreausmooth} \label{lem:moreausmooth}
The function $f$ is convex, $\ell$-Lipschitz, and $\eta$-smooth, with $\ell \leq L$ and $\eta \leq H$.
\end{restatable}

Furthermore, optimizing $f$ is equivalent to ``finding'' the vectors $v_1,\dots,v_{D+1}$. In particular, until a point that has a substantial inner product with all of $v_1,\dots,v_{D+1}$ is found, the algorithm will remain far from the minimum:
\begin{restatable}{lemma}{gradonlysubopt} \label{lem:gradonlysubopt}
For any $H,L > 0$, $D \geq 1$, and orthonormal $v_1,...,v_{D+1}$, for any $x$ with $\abs{v_{D+1}^\top x}\leq \frac{\ell}{\eta}$
\[
f(x) - \min_{x:\norm{x}\leq 1} f(x) \geq \min\set{\frac{L}{2\sqrt{D+1}},\ \frac{H}{20(D+1)^2}}
\]
\end{restatable}

The function also has the property that if $x$ has a small inner product with $v_t,\dots,v_{D+1}$, then the gradient oracle will reveal little information about $f$ when queried at $x$:
\begin{restatable}{lemma}{gradspan} \label{lem:gradspan}
For any $x$ with $\abs{\inner{x}{v_r}} \leq \frac{\ell}{\eta}$ for all $r \geq t$,
both the function value $f(x)$ and gradient $\nabla f(x)$ can be calculated from $v_1,\dots,v_t$ only.
\end{restatable}
In Appendix \ref{sec:vectorguessing}, we studied the situation where orthonormal $v_1,\dots,v_{D+1}$ are chosen uniformly at random and a sequence of sets of vectors $X_1,\dots,X_{D+1}$ are generated as
\begin{equation}\label{eq:decomposition1}
X_t\parens{V} = X_t\parens{V_{<t}\indicator{G'_{<t}} +  V\indicator{\lnot G'_{<t}}}
\end{equation}
where
\begin{equation}
G'_{<t} = \left\llbracket  \forall r < t,\ \forall x \in X_r,\ \forall j\geq r\ \abs{\inner{x}{v_j}} \leq \frac{c}{2} \right\rrbracket
\end{equation}
Take $c = \frac{2\ell}{\eta}$ and consider the dependency graph. Let $X_1$ be the set of queries made in vertices at depth 1 in the graph (i.e.~they have no parents). Let $X_2$ be the set of queries made in vertices at depth 2 in the graph (i.e.~their parents correspond to the queries in $X_1$). Continue in this way for each $t \leq D$, and let $X_{D+1} = \set{\hat{x}}$ correpond to the algorithm's output, which is allowed to depend on all queries and oracle responses in the graph, and thus has depth $D+1$.

Supposing $G'_{<t}$, for all queries $x \in X_1\cup\dots\cup X_{t-1}$ and for all $r \geq t-1$ we have $\abs{\inner{x}{v_r}} \leq \frac{c}{2} = \frac{\ell}{\eta}$. Therefore, by Lemma \ref{lem:gradspan} all of the function evaluations and gradients returned by the stochastic gradient oracle are calculable from $v_1,\dots,v_{t-1}$ only. Therefore, all of the queries in $X_t$ are a deterministic function of $V_{<t}$ (since we are currently considering only deterministic optimization algorithms), so $X_t$ satisfies the required decomposition property \eqref{eq:decomposition1}. Finally, the queries are required to be in the domain of $f$, thus they will have norm bounded by $1$.

Therefore, by Lemma \ref{lem:keylowerboundlemma}, when the dimension
\begin{align}
m 
&\geq D+1+N+\max\set{32N^2,\ 200\parens{D+1}^3(1+\sqrt{2N})^2}\log\parens{2(D+1)^2N} 
\end{align}
with probability $1/2$, all $x \in X_1\cup\dots\cup X_{t+1}$ including the algorithm's output $\hat{x}$ satsify $\abs{\inner{x}{v_{D+1}}} \leq \frac{\ell}{\eta}$ so by Lemma \ref{lem:gradonlysubopt}
\begin{equation}
f(\hat{x}) - \min_{x:\norm{x}\leq 1} f(x) \geq \min\set{\frac{L}{2\sqrt{D+1}},\ \frac{H}{20(D+1)^2}}
\end{equation}
Therefore, by Yao's minimax principle for any randomized algorithm $\A$
\begin{multline}
\max_{V}\mathbb{E}_{\hat{X}\sim\A}\left[ f(\hat{X})\right] - \min_{x:\norm{x}\leq 1} f(x) \geq 
\min_{\textrm{deterministic }\A}\mathbb{E}_{V}\left[ f(\hat{x}) \right] - \min_{x:\norm{x}\leq 1} f(x)  \\ \geq \min\set{\frac{L}{4\sqrt{D+1}},\ \frac{H}{40(D+1)^2}}
\end{multline}
The statistical term $\frac{L}{8\sqrt{N}}$ follows from Lemma \ref{lem:statisticalLB}.
\end{proof}

\subsection{Deferred proofs} \label{sec:gradonlydeferred}
\moreausmooth*
\begin{proof}
Since $\tilde{f}$ is the maximum of $\ell$-Lipschitz affine functions, it is convex and $\ell$-Lipschitz. Furthermore, by Proposition 12.29 \cite{bauschke2017convex}, $f$, the $\eta$-Moreau Envelope of $\tilde{f}$ is $\eta$-smooth and
\begin{equation}
\nabla f(x) = \eta\parens{x - \argmin_{y} \tilde{f}(y) + \frac{\eta}{2}\norm{y - x}^2}
\end{equation}
The minimizing $y$ satisfies that $\eta(x-y) \in \partial \tilde{f}(y)$ (where $\partial \tilde{f}(y)$ denotes the set of subgradients of $\tilde{f}$ at $y$), and since $\tilde{f}$ is $\ell$-Lipschitz this implies that $\norm{\nabla f(x)} \leq \ell$.
\end{proof}

\gradonlysubopt*
\begin{proof}
First
\begin{align}
\min_{x:\norm{x}\leq 1} f(x) 
\leq f\parens{-\sum_{r=1}^{D+1} \frac{v_r}{\sqrt{D+1}}}
\leq \tilde{f}\parens{-\sum_{r=1}^{D+1} \frac{v_r}{\sqrt{D+1}}}
\leq -\frac{\ell}{\sqrt{D+1}}
\label{eq:gradonlyksubopt}
\end{align}

Now, for an arbitrary point $x$ such that $\abs{v_{D+1}^\top x}\leq\frac{\ell}{\eta} = \frac{1}{10(D+1)^{1.5}}$, consider
\begin{align}
y^* = \textrm{prox}_{\tilde{f}}(x,\eta) = \argmin_{y}\set{\max_{1\leq r \leq D+1} \parens{\ell v_r^\top y - \frac{5\ell^2(r-1)}{\eta}} + \frac{\eta}{2}\norm{y-x}^2}
\end{align}
Since $y^*$ is the minimizer, $\eta(x-y^*) \in \partial \tilde{f}(y^*)$
and since $\tilde{f}$ is $\ell$-Lipschitz, $\norm{x-y^*} \leq \frac{\ell}{\eta}$.
Thus $v_{D+1}^\top y^* \geq -\frac{2\ell}{\eta}$ and
\begin{align}
f(x) 
&= \tilde{f}(y^*) + \frac{\eta}{2}\norm{y^*-x}^2 \\
&= \max_{1\leq r \leq D+1} \parens{\ell v_r^\top y^* - \frac{5\ell^2(r-1)}{\eta}} + \frac{\eta}{2}\norm{y^*-x}^2 \\
&\geq \ell v_{D+1}^\top y^* - \frac{5\ell^2D}{\eta} \\
&\geq -\frac{2\ell^2}{\eta} - \frac{5\ell^2D}{\eta} \\
&\geq - \frac{5\ell^2(D+1)}{\eta} \label{eq:gradonlytsubopt}
\end{align}
Combining \eqref{eq:gradonlyksubopt} and \eqref{eq:gradonlytsubopt}, 
for any $x$ such that $\abs{v_{D+1}^\top x}\leq\frac{\ell}{\eta}$
\begin{align}
f(x) - \min_{x:\norm{x}\leq 1} f(x)
\geq \frac{\ell}{\sqrt{D+1}} -\frac{5\ell^2(D+1)}{\eta} 
= \min\set{\frac{L}{2\sqrt{D+1}},\ \frac{H}{20(D+1)^2}}
\end{align}
\end{proof}

\gradspan*
\begin{proof}
Let $x$ be a point such that $\abs{v_r^\top x} \leq \frac{\ell}{\eta}$ for all $r \geq t$.
By Proposition 12.29 \cite{bauschke2017convex}
\begin{equation}\label{eq:gradspanystar}
\nabla f(x) = \eta\parens{x - \textrm{prox}_{\tilde{f}}(x,\eta)}
\end{equation}
Since $f$ is $\ell$-Lipschitz (Lemma \ref{lem:moreausmooth}), $\norm{x-\textrm{prox}_{\tilde{f}}(x,\eta)} \leq \frac{\ell}{\eta}$. Consequently, for $y^* = \textrm{prox}_{\tilde{f}}(x,\eta)$ we have $\abs{v_r^\top y^*} \leq \frac{2\ell}{\eta}$ for all $r \geq t$.
Furthermore,   
\begin{equation} \label{eq:gradonlyargmax}
\nabla f(x) = \eta(x - y^*) \in \textrm{conv}\set{\ell v_r\,:\, r \in \argmax_{1\leq r \leq D+1} \parens{\ell v_r^\top y^* - \frac{5\ell^2(r-1)}{\eta}}}
\end{equation}
For any $r > t$
\begin{equation}\label{eq:gradonlymaxlow}
\ell v_r^\top y^* - \frac{5\ell^2(r-1)}{\eta} \leq \frac{2\ell^2}{\eta} - \frac{5\ell^2(r-1)}{\eta} = - \frac{5\ell^2\parens{r-\frac{7}{5}}}{\eta}
\end{equation}
Whereas
\begin{equation}\label{eq:gradonlymaxhigh}
\ell v_{t}^\top y^* - \frac{5\ell^2(t-1)}{\eta} \geq -\frac{2\ell^2}{\eta} - \frac{5\ell^2(t-1)}{\eta} = - \frac{2\ell^2\parens{t-\frac{3}{5}}}{\eta}
\end{equation}
For any $r > t$ \eqref{eq:gradonlymaxlow} is less than \eqref{eq:gradonlymaxhigh}, thus no $r > t$ can be in the $\argmax$ in \eqref{eq:gradonlyargmax}. Therefore, using only $v_1,\dots,v_t$ we can calculate
\begin{align}
f(x) 
&= \inf_{y}\set{\max_{1\leq r \leq D+1} \parens{\ell v_r^\top y - \frac{5\ell^2(r-1)}{\eta}} + \frac{\eta}{2}\norm{y-x}^2} \\
&= \inf_{y}\set{\max_{1\leq r \leq t} \parens{\ell v_r^\top y - \frac{5\ell^2(r-1)}{\eta}} + \frac{\eta}{2}\norm{y-x}^2} \\
\end{align}
and
\begin{align}
\textrm{prox}_{\tilde{f}}(x,\eta)
&= \argmin_{y}\set{\max_{1\leq r \leq D+1} \parens{\ell v_r^\top y - \frac{5\ell^2(r-1)}{\eta}} + \frac{\eta}{2}\norm{y-x}^2} \\
&= \argmin_{y}\set{\max_{1\leq r \leq t} \parens{\ell v_r^\top y - \frac{5\ell^2(r-1)}{\eta}} + \frac{\eta}{2}\norm{y-x}^2}
\end{align}
from which we get $\nabla f(x) = \eta(x - \textrm{prox}_{\tilde{f}}(x,\eta))$.
\end{proof}

\section{Proof of Theorem \ref{thm:proxLB}}\label{sec:proxappendix}
\proxLB*
\begin{proof}
Without loss of generality, assume $B = 1$, the lower bound can be proven for other values of $B$ by scaling inputs to our construction by $1/B$. Let 
\begin{equation}\label{eq:LHlbparameters}
\eta = \min\set{H,\ 2LD} \qquad\qquad
\gamma = \frac{4L}{\eta\sqrt{2D}} \qquad\qquad
a = 2c = \frac{1}{\sqrt{8D^3}} 
\end{equation}
Define the following scalar function
\begin{align}
\phi_c(z) &= 
\begin{cases} 
0 & \abs{z} \leq c \\
2(\abs{z} - c)^2 & c < \abs{z} \leq 2c \\
z^2 - 2c^2 & 2c < \abs{z} \leq \gamma \\
2\gamma\abs{z} - \gamma^2 - 2c^2 & \abs{z} > \gamma
\end{cases} 
\end{align}
It is straightforward to confirm that $\phi_c$ is convex, $2\gamma$-Lipschitz continuous, and $4$-smooth.
Let $\Pcal$ be the uniform distribution over $\set{1,2}$. Let $v_1,v_2,\dots,v_{2D}$ be a set of orthonormal vectors drawn uniformly at random and define
\begin{align}
f(x;1) &=  
\frac{\eta}{8}\parens{-2av_{1}^\top x + \phi_c\parens{v_{2D}^\top x} + \sum_{r=3,5,7,...}^{2D-1} \phi_c\parens{v_{r-1}^\top x - v_{r}^\top x}}
\\
f(x;2) &= 
\frac{\eta}{8}\parens{\sum_{r=2,4,6,...}^{2D} \phi_c\parens{v_{r-1}^\top x - v_{r}^\top x}} \\
F(x) &= \mathbb{E}_{z\sim\Pcal}\left[f(x;z)\right] 
= \frac{1}{2}\parens{f(x;1) + f(x;2)} \\
&= \frac{\eta}{16}\parens{-2av_{1}^\top x + \phi_c\parens{v_{2D}^\top x} + \sum_{r=2}^{2D} \phi_c\parens{v_{r-1}^\top x - v_{r}^\top x}}
\end{align}
The random choice of $V$ determines a distribution over functions $f(\cdot;1)$ and $f(\cdot;2)$. We will lower bound the expectation (over $V$) of the suboptimality of any deterministic algorithm's output, and then apply Yao's minimax principle.

First, we show that the functions $f(\cdot;1)$ and $:= f(\cdot;2)$ are convex, $L$-Lipschitz, and $H$-smooth:
\begin{restatable}{lemma}{convexlipschitzsmooth} \label{lem:convexlipschitzsmooth}
For any $H,L \geq 0$, $D \geq 1$, and orthonormal $v_1,...,v_{2D}$, and with $\eta$, $\gamma$, $a$, and $c$ chosen as in \eqref{eq:LHlbparameters},
$f(\cdot;1)$ and $f(\cdot;2)$ are convex, $L$-Lipschitz, and $H$-smooth.
\end{restatable}

Next, we show that optimizing $F$ is equivalent to ``finding'' a large number of the vectors $v_1,\dots,v_{2D}$:
\begin{restatable}{lemma}{suboptimality} \label{lem:LHlbsuboptimality}
For any $H,L \geq 0$, $D \geq 1$, and orthonormal $v_1,...,v_{2D}$, and with $\eta$, $\gamma$, $a$, and $c$ chosen as in \eqref{eq:LHlbparameters}, for any $x$ such that $\abs{v_r^\top x}\leq\frac{c}{2}$ for all $r > D$
\[
F(x) - \min_{x:\norm{x}\leq 1} F(x) \geq \min\set{\frac{L}{32D},\ \frac{H}{64D^2}}
\]
\end{restatable}

Next, we show that at any point $x$ such that $\abs{v_r^\top x}\leq\frac{c}{2}$ for all $r \geq t$, the function value, gradient, and prox of $f(\cdot;1)$ and $f(\cdot;2)$ at $x$ are calculable using $v_1,\dots,v_t$ only:
\begin{restatable}{lemma}{gradproxoraclespan} \label{lem:gradproxoraclespan} 
For any $x$ such that $\abs{v_r^\top x}\leq\frac{c}{2}$ for all $r \geq t$, and any $\beta \geq 0$ the function values, gradients, and proxs
$f(x;1)$, $f(x;2)$, $\nabla f(x;1)$, $\nabla f(x;2)$, $\textrm{prox}_{f(\cdot,1)}(x,\beta)$, and $\textrm{prox}_{f(\cdot,2)}(x,\beta)$ are calculable using $\beta,x,v_1,\dots,v_t$ only.
\end{restatable}
In Appendix \ref{sec:vectorguessing}, we studied the situation where orthonormal $v_1,\dots,v_{2D}$ are chosen uniformly at random and a sequence of sets of vectors $X_1,\dots,X_{2D}$ are generated as
\begin{equation}\label{eq:decomposition2}
X_t\parens{V} = X_t\parens{V_{<t}\indicator{G'_{<t}} + V\indicator{\lnot G'_{<t}}}
\end{equation}
where
\begin{equation}
G'_{<t} = \left\llbracket  \forall r < t,\ \forall x \in X_r,\ \forall j\geq r\ \abs{\inner{x}{v_j}} \leq \frac{c}{2} \right\rrbracket
\end{equation}
Consider the dependency graph, and let $X_1$ be the set of queries made in vertices at depth 1 in the graph (i.e.~they have no parents). Let $X_2$ be the set of queries made in vertices at depth 2 in the graph (i.e.~their parents correspond to the queries in $X_1$). Continue in this way for each $t \leq D$, and then let $X_{D+1} = \set{\hat{x}}$ correpond to the output of the optimization algorithm, which for now is deterministic. 

Suppose $G'_{<t}$. Then for all of the queries $x \in X_1\cup\dots\cup X_{t-1}$ and for all $r \geq t-1$ we have $\abs{\inner{x}{v_r}} \leq \frac{c}{2}$. Therefore, by Lemma \ref{lem:gradproxoraclespan} the function values, gradients, and proxs of $f(\cdot;1)$ and $f(\cdot;2)$ are calculable based only on the query points and $v_1,\dots,v_{t-1}$.  Therefore, all of the queries in $X_t$ are a deterministic function of $V_{<t}$ only so $X_t$ satisfies the required decomposition property \eqref{eq:decomposition2}. Finally, the queries are required to be in the domain of $f$, thus they will have norm bounded by $B$.

Therefore, by Lemma \ref{lem:keylowerboundlemma} for
\begin{equation}
m \geq 2D + N + \max\set{32N^2,\ 128B^2D^3(1+\sqrt{2N})^2}\log\parens{8D^2N}
\end{equation}
with probability $1/2$ for every $x\in X_1\cup\dots\cup X_{D+1}$ which includes $\hat{x}$,  $\abs{\inner{x}{v_{r}}} \leq \frac{c}{2}$ for $r > D$, so by Lemma \ref{lem:LHlbsuboptimality}
\begin{equation}
f(\hat{x}) - \min_{x:\norm{x}\leq 1} f(x) \geq \min\set{\frac{L}{32D},\ \frac{H}{64D^2}}
\end{equation}
Therefore,
\begin{equation}
\min_{\textrm{deterministic }\A}\mathbb{E}_{V}\left[
f(\hat{x}) - \min_{x:\norm{x}\leq 1} f(x) \right] \geq \min\set{\frac{L}{64D},\ \frac{H}{128D^2}}
\end{equation}
so by Yao's minimax principle, for any randomized algorithm $\A$
\begin{equation}
\max_{V}\mathbb{E}_{\hat{X}\sim\A}\left[
f(\hat{X}) - \min_{x:\norm{x}\leq 1} f(x) \right] \geq \min\set{\frac{L}{64D},\ \frac{H}{128D^2}}
\end{equation}
The statistical term $\frac{LB}{8\sqrt{N}}$ follows from Lemma \ref{lem:statisticalLB}.
\end{proof}

\subsection{Deferred proof}
\convexlipschitzsmooth*
\begin{proof}
The functions $f(\cdot;1)$ and $f(\cdot;2)$ are a sum of linear functions and $\phi_c$, which is convex; therefore both are convex. 
Also, the scalar function $\phi_c$ is $2\gamma$-Lipschitz, so 
\begin{align}
\norm{\nabla f(x;1)}^2 
&= \norm{\frac{\eta}{8}\parens{-2av_{1} + \phi_c'\parens{v_{2D}^\top x}v_{2D} + \sum_{r=3,5,7,...}^{2D-1} \phi'_c\parens{v_{r-1}^\top x - v_{r}^\top x}\parens{v_{r-1}-v_r}}}^2 \\
&\leq \frac{\eta^2\parens{a^2 + (2D-1)\gamma^2}}{16} \leq \frac{2D\eta^2\gamma^2}{16} = L^2
\end{align}
where we used that $a = \frac{1}{\sqrt{8D^3}} < \gamma = \frac{4L}{\eta\sqrt{2D}}$. Similarly,
\begin{align}
\norm{\nabla f(x;2)}^2 
= \norm{\frac{\eta}{8}\parens{\sum_{r=2,4,6,...}^{2D} \phi'_c\parens{v_{r-1}^\top x - v_{r}^\top x}\parens{v_{r-1}-v_r}}}^2 
\leq \frac{2D\eta^2\gamma^2}{16} = L^2
\end{align}
Therefore, $f(\cdot;1)$ and $f(\cdot;2)$ are $L$-Lipschitz. Furthermore, since $\phi_c$ is $4$-smooth,
\begin{equation}
\abs{v_i^\top \nabla^2f(x;1) v_j} \leq
\begin{cases}
\frac{\eta}{2} & \abs{i-j} \leq 1 \\
0 & \abs{i-j} > 1
\end{cases} 
\mathand 
\abs{v_i^\top \nabla^2f(x;2) v_j} \leq
\begin{cases}
\frac{\eta}{2} & \abs{i-j} \leq 1 \\
0 & \abs{i-j} > 1
\end{cases} 
\end{equation}
therefore, the maximum eigenvalue of $\nabla^2 f(\cdot;1)$ and $\nabla^2 f(\cdot;2)$ is at most $\eta \leq H$.
\end{proof}

\suboptimality*
\begin{proof}
First, we upper bound $\min_{x:\norm{x}\leq 1} F(x)$.
Recalling that $a = \frac{1}{\sqrt{8D^3}}$, define
\begin{align}
x^* &= a \sum_{r=1}^{2D} (2D+1-r) v_r \\
\norm{x^*}^2 &= \frac{1}{8D^3}\parens{\frac{2D(2D+1)(4D+1)}{6}} \leq 1
\end{align}
For this $x^*$, $v_{r-1}^\top x^* - v_{r}^\top x^* = v_{2D}^\top x^* = a$ and with our choice of parameters \eqref{eq:LHlbparameters}, $2c = a < \gamma$, so that $\phi_c'(a) = 2a$, thus
\begin{align}
\nabla F(x^*) &= 
\frac{\eta}{16}\parens{-2av_1 + \phi'_c\parens{v_{2D}^\top x^*}v_{2D} + \sum_{r=2}^{2D} \phi'_c\parens{v_{r-1}^\top x^* - v_{r}^\top x^*}(v_{r-1}-v_r)}
\end{align}
thus,
\begin{align}
\nabla F(x^*)^\top v_1 &= -2a + \phi_c'\parens{a} = 0 \\
\nabla F(x^*)^\top v_r &= -\phi'_c\parens{a} + \phi'_c\parens{a} = 0 &\qquad 2 \leq r \leq 2D-1 \\
\nabla F(x^*)^\top v_{2D} &= -\phi'_c\parens{a} + \phi'_c\parens{a} = 0
\end{align} 
Since $\norm{x^*} \leq 1$ and $\nabla F(x^*) = 0$, we conclude
\begin{align}
\min_{x:\norm{x}\leq 1} F(x) 
&= F(x^*) = \frac{\eta}{16}\parens{-2Da^2 - 4Dc^2)} = -\frac{\eta D a^2}{4} = -\frac{\eta}{32D^2} \label{eq:ksuboptimality}
\end{align}

Let $X_D = \set{x:\norm{x}\leq 1,\ \abs{v_r^\top x} \leq \frac{c}{2}\ \forall r > D}$.
We will now lower bound
\begin{equation} \label{eq:tconstrainedminimization}
\min_{x\in X_D}F(x) = \min_{x:\norm{x}\leq 1} F(x) \quad \textrm{s.t.}\quad \abs{v_r^\top x} \leq \frac{c}{2} \quad \forall r > D
\end{equation}

Introducing dual variables $\lambda_{D+1},...,\lambda_{2D} \geq 0$, solving \eqref{eq:tconstrainedminimization} amounts to finding an $x \in X_D$ and a set of non-negative $\lambda$s such that $\nabla F(x) = -\sum_{r=D+1}^{2D} \lambda_r\, \textrm{sign}\parens{v_r^\top x} v_r$ and such that $\lambda_r\parens{v_r^\top x - \frac{c}{2}} = 0$ for each $r$.
Let
\begin{align}
x_D &= \sum_{r=1}^{D+1}\parens{a\parens{D+1-r} + \frac{c}{2}}v_r,\quad \lambda_{D+1} = 2a,\ \quad \lambda_{D+2} = \dots = \lambda_k = 0
\end{align}
Since $a\parens{D+1-r} + \frac{c}{2} < a\parens{2D+1-r}$ for $r \leq D+1$ and $\norm{x^*} \leq 1$ it follows that $\norm{x_D} \leq 1$. Furthermore, since $v_{r-1}^\top x_D - v_{r}^\top x_D = a$ for $2 \leq r \leq D+1$ and $2c = a < \gamma$, the gradient
\begin{align}
\nabla F(x_D)^\top v_1 &= -2a + \phi_c'\parens{a} = 0 \\
\nabla F(x_D)^\top v_r &= -\phi'_c\parens{a} + \phi'_c\parens{a} = 0 &2 \leq r \leq D \\
\nabla F(x_D)^\top v_{D+1} &= -\phi'_c\parens{a} + \phi'_c\parens{\frac{c}{2}} = -2a = -\lambda_{D+1} \\
\nabla F(x_D)^\top v_r &= 0 = -\lambda_{r}  & D+2 \leq r \leq 2D
\end{align} 
Therefore,
\begin{align}
\min_{x\in X_D} F(x) = F(x_D)
&= \frac{\eta}{16}\parens{-Da^2 - ac - 2Dc^2} = -\frac{\eta(3D+1)a^2}{32}
= -\frac{\eta(3D+1)}{256D^3} \label{eq:tsuboptimality}
\end{align}
Combining \eqref{eq:ksuboptimality} and \eqref{eq:tsuboptimality} we have that
\begin{multline}
\min_{x\in X_D} F(x) - \min_{x:\norm{x}\leq 1} F(x) = F(x_D) - F(x^*) \\
= \frac{\eta}{32D^2} - \frac{\eta(3D+1)}{256D^3} 
\geq \frac{\eta}{32D^2} - \frac{\eta}{64D^2} 
= \min\set{\frac{L}{32D},\ \frac{H}{64D^2}}
\end{multline}
\end{proof}

\gradproxoraclespan*
\begin{proof}
Suppose that $x$ is a point such that $\abs{v_r^\top x} \leq \frac{c}{2}$ for all $r \geq t$, and $\beta \geq 0$. Therefore, $\phi_c\parens{v_{r-1}^\top x - v_{r}^\top x} = 0$ for $r > t$ so
\begin{align}
f(x;1) 
&= \frac{\eta}{8}\parens{-2av_{1}^\top x + \phi_c\parens{v_{2D}^\top x} + \sum_{r=3,5,7,...}^{2D-1} \phi_c\parens{v_{r-1}^\top x - v_{r}^\top x}} \\
&= \frac{\eta}{8}\parens{-2av_{1}^\top x + \sum_{r=3,5,7,...}^{t} \phi_c\parens{v_{r-1}^\top x - v_{r}^\top x}} \\
f(x;2) 
&= \frac{\eta}{8}\parens{\sum_{r=2,4,6,...}^{2D} \phi_c\parens{v_{r-1}^\top x - v_{r}^\top x}} \\
&= \frac{\eta}{8}\parens{\sum_{r=2,4,6,...}^{t} \phi_c\parens{v_{r-1}^\top x - v_{r}^\top x}}
\end{align}
Thus both $f(x;1)$ and $f(x;2)$ can be calculated from $x,v_1,\dots,v_t$ only. Similarly, $\phi'_c\parens{v_{r-1}^\top x - v_{r}^\top x} = 0$ for $r > t$ so
\begin{align}
\nabla f(x;1) 
&= \frac{\eta}{8}\parens{-2av_{1} + \phi_c'\parens{v_{2D}^\top x}v_{2D} + \sum_{r=3,5,7,...}^{2D-1} \phi'_c\parens{v_{r-1}^\top x - v_{r}^\top x}\parens{v_{r-1}-v_r}} \\
&= \frac{\eta}{8}\parens{-2av_{1} + \sum_{r=3,5,7,...}^{t} \phi'_c\parens{v_{r-1}^\top x - v_{r}^\top x}\parens{v_{r-1}-v_r}} \\
\nabla f(x;2) 
&= \frac{\eta}{8}\parens{\sum_{r=2,4,6,...}^{2D} \phi'_c\parens{v_{r-1}^\top x - v_{r}^\top x}\parens{v_{r-1}-v_r}} \\
&= \frac{\eta}{8}\parens{\sum_{r=2,4,6,...}^t \phi'_c\parens{v_{r-1}^\top x - v_{r}^\top x}\parens{v_{r-1}-v_r}}
\end{align}
Thus $\nabla f(x;1)$ and $\nabla f(x;2)$ can also be calculated from $x,v_1,\dots,v_t$ only.

Now, we consider the proxs at such a point $x$.
Let $t' = t$ if $t$ is odd, and $t' = t-1$ if $t$ is even. Let $P$ be the projection operator onto $S = \spn{v_1,\dots,v_{t'}}$ and let $P^\perp$ be the projection onto the orthogonal subspace, $S^\perp$. 
Then, since $f(x;1) = f(Px;1) + f(P^\perp x;1)$, we can decompose the prox:
\begin{align}
&\textrm{prox}_{f(\cdot;1)}(x,\beta) \nonumber\\
&= \argmin_{y} f(y;1) + \frac{\beta}{2}\norm{y-x}^2 \\
&= \argmin_{y_1 \in S, y_2 \in S^\perp} f(y_1;1) + f(y_2;1) + \frac{\beta}{2}\parens{\norm{y_1 - Px}^2 + \norm{y_2 - P^\perp x}^2}  \\
&= \argmin_{y_1 \in S} \frac{\eta}{8}\parens{-2av_{1}^\top y_1 + \sum_{r=3,5,7,...}^{t'} \phi_c\parens{v_{r-1}^\top y_1 - v_{r}^\top y_1}} + \frac{\beta}{2}\norm{y_1 - Px}^2 \\
&+ \argmin_{y_2 \in S^\perp} \frac{\eta}{8}\parens{\phi_c\parens{v_{2D}^\top y_2} + \sum_{r=t'+2,t'+4,\dots}^{2D-1} \phi_c\parens{v_{r-1}^\top y_2 - v_{r}^\top y_2}} + \frac{\beta}{2}\norm{y_2 - P^\perp x}^2 \\
&= P^\perp x + \argmin_{y_1 \in S} \frac{\eta}{8}\parens{-2av_{1}^\top y_1 + \sum_{r=3,5,7,...}^{t'} \phi_c\parens{v_{r-1}^\top y_1 - v_{r}^\top y_1}} + \frac{\beta}{2}\norm{y_1 - Px}^2
\end{align}
Where we used that $\abs{v_r^\top P^\perp x}=\abs{v_r^\top x}\leq\frac{c}{2}$ for all $r > t'$, so setting $y_2 = P^\perp x$ achieves the minimum since every term in the expression is zero and function is non-negative. The vector $P^\perp x$ is calculable from $x,v_1,\dots,v_{t'} \subseteq x,v_1,\dots,v_t$, and similarly the second term is a minimization depends only on $\beta,x,v_1,\dots,v_{t'} \subseteq \beta,x,v_1,\dots,v_{t}$. A nearly identical argument shows that $\prox_{f(\cdot;2)}(x,\beta)$ has the same property.
\end{proof}

\section{Statistical term}\label{sec:statisticalLB}
\begin{restatable}{lemma}{statisticalLB} \label{lem:statisticalLB}
For any $L,B > 0$, there exists a distribution $\Pcal$, and an $L$-Lipschitz, $0$-smooth function $f$ defined on $[-B,B]$ such that the output $\hat{x}$ of any potentially randomized optimization algorithm which accesses a stochastic gradient or prox oracle at most $N$ times satisfies
\[
\mathbb{E}_{\hat{X}\sim\A}\left[ \mathbb{E}_{z\sim\Pcal}\left[ f(\hat{X};z) \right] - \min_{\abs{x}\leq B}\mathbb{E}_{z\sim\Pcal}\left[ f(x;z) \right]  \right] \geq \frac{LB}{8\sqrt{N}}
\] 
\end{restatable}
\begin{proof}
Let $\epsilon > 0$ and $p \sim \textrm{Uniform}\set{p_1,p_{-1}}$ where $p_1 = \frac{1+\epsilon}{2}$ and $p_{-1} = \frac{1-\epsilon}{2}$. Define $\Pcal_p$ as
\begin{equation}
\mathbb{P}_{\Pcal_p}\left[ Z = 1 \right] = 1 - \mathbb{P}_{\Pcal_p}\left[ Z = -1 \right] = p
\end{equation} 
Then, let $f(x;z) = zLx$, so $\mathbb{E}_{z\sim\Pcal_p}\left[ f(x;z) \right] = (2p-1)Lx$. When $p=p_1$, $(2p-1) > 0$ so the minimizer is $x = -B$ with value $-LB(2p-1) = -LB\epsilon$, and when $p=p_{-1}$, $(2p-1) < 0$ so the minimizer is $x = B$, also with value $-LB\epsilon$. Furthermore, if $p = p_1$ and $x \geq 0$ then it is at least $LB\epsilon$-suboptimal, and if $p = p_2$ and $x \leq 0$ then it is also at least $LB\epsilon$-suboptimal. 

Now consider any deterministic optimization algorithm which accesses the gradient or prox oracle $N$ times. Each gradient or prox oracle response can be simulated using a single $z \sim \Pcal_p$, so the algorithm's output is $\hat{x} = \hat{x}(z_1,\dots,z_N) \in [-B,B]$. Consider
\begin{multline}
\mathbb{E}_{p \sim \textrm{Uniform}\set{p_1,p_{-1}},z\sim\Pcal_p}\left[ (2p-1)L\hat{x}(z_1,\dots,z_N) \,\middle|\, z_1,\dots,z_N \right] \\
\geq LB\epsilon \mathbb{P}_{p \sim \textrm{Uniform}\set{p_1,p_{-1}},z\sim\Pcal_p}\left[ \textrm{sign}(\hat{x}(z_1,\dots,z_N)) \neq \textrm{sign}(2p-1)\,\middle|\, z_1,\dots,z_N \right]
\end{multline}
Furthermore, the Bayes optimal estimate $\hat{x}$ of $\textrm{sign}(2p-1)$ is
\begin{equation}
\hat{x}(z_1,\dots,z_N) =
\begin{cases}
1 & \frac{1}{N}\sum_{i=1}^N z_i \geq 0 \\
-1 & \frac{1}{N}\sum_{i=1}^N z_i < 0 
\end{cases}
\end{equation}
so
\begin{align}
&\mathbb{P}_{p \sim \textrm{Uniform}\set{p_1,p_{-1}},z\sim\Pcal_p}\left[ \textrm{sign}(\hat{X}(z_1,\dots,z_N)) \neq \textrm{sign}(2p-1)\,\middle|\, z_1,\dots,z_N \right] \nonumber\\
&\geq \mathbb{P}_{p \sim \textrm{Uniform}\set{p_1,p_{-1}},z\sim\Pcal_p}\left[ \abs{\frac{1}{N}\sum_{i=1}^N z_i - (2p-1)} \geq \epsilon \right] \\
&= \mathbb{P}_{z\sim\Pcal_{p_{-1}}}\left[ \abs{\frac{1}{N}\sum_{i=1}^N z_i - \epsilon} \geq \epsilon \right]
\end{align} 
This simply requires lower bounding the tail of the Binomial$\parens{N,\frac{1-\epsilon}{2}}$ distribution. Using Theorem 2.1 in \cite{slud1977distribution}, 
\begin{align}
\mathbb{P}_{z\sim\Pcal_{p_{-1}}}\left[ \abs{\frac{1}{N}\sum_{i=1}^N z_i - \epsilon} \geq \epsilon \right]
&\geq 1 - \Phi\parens{\frac{\epsilon N}{\sqrt{N(1+\epsilon)(1-\epsilon)}}} = 1 - \Phi\parens{\frac{\epsilon\sqrt{N}}{\sqrt{1-\epsilon^2}}}
\end{align}
where $\Phi$ is the distribution function of the standard normal. Let $\epsilon = \frac{1}{2\sqrt{N}}$, then 
$\frac{\epsilon\sqrt{N}}{\sqrt{1-\epsilon^2}} < \frac{3}{5}$ and
\begin{equation}
\mathbb{P}_{z\sim\Pcal_{p_{-1}}}\left[ \abs{\frac{1}{N}\sum_{i=1}^N z_i - \epsilon} \geq \epsilon \right]
\geq 1 - \Phi\parens{\frac{3}{5}} \geq \frac{1}{4}
\end{equation}
Therefore, we conclude that 
\begin{equation}
\mathbb{E}_{p \sim \textrm{Uniform}\set{p_1,p_{-1}},z\sim\Pcal_p}\left[ (2p-1)L\hat{x}(z_1,\dots,z_N) \,\middle|\, z_1,\dots,z_N \right] \geq \frac{LB\epsilon}{4} = \frac{LB}{8\sqrt{N}}
\end{equation}
Therefore, by Yao's minimax principle, for any randomized algorithm $\A$
\begin{equation}
\max_{p\in\set{p_1,p_{-1}}}\mathbb{E}_{\hat{X}\sim\A}\left[ \mathbb{E}_{z\sim\Pcal_p}\left[ f(\hat{X};z) \right] - \min_{x}\mathbb{E}_{z\sim\Pcal_p}\left[ f(x;z) \right] \right] \geq \frac{LB}{8\sqrt{N}}
\end{equation}
\end{proof}

\section{Supplement to Section \ref{sec:examples}}
\subsection{Smoothed accelerated mini-batch SGD}\label{sec:smoothedambsgd}
Smoothed accelerated mini-batch SGD is the composition of two ingredients. First, we approximate the non-smooth $f$ with a smooth surrogate, and then perform accelerated mini-batch SGD on the surrogate \citep{lan2012optimal,cotter2011better}. In particular, we use the $\beta$-Moreau envelope $f^{(\beta)}$ of $f$:
\begin{equation}
f^{(\beta)}(x;z) = \inf_{y} f(y;z) + \frac{\beta}{2}\norm{y-x}^2
\end{equation}
Since $f$ is $L$-Lipschitz, $f^{(\beta)}$ has the following properties (Proposition 12.29 \cite{bauschke2017convex}): 
\begin{enumerate}
\item $f^{(\beta)}$ is $\beta$-smooth
\item $\nabla f^{(\beta)}(x;z) = \beta(x - \textrm{prox}_{f(\cdot;z)}(x,\beta))$
\item $f^{(\beta)}(x;z) \leq f(x;z) \leq f^{(\beta)}(x;z) + \frac{L^2}{2\beta}$ for all $x$
\end{enumerate}
We use the prox oracle to execute \ambsgd on the $L$-Lipschitz and $\beta$-smooth $f^{(\beta)}$, with updates
\begin{align}
w_t &= \alpha y_t + (1-\alpha) x_t \\
y_{t+1} &= w_t - \frac{\eta}{M} \sum_{i=1}^M \beta\parens{w_t - \prox_{f(\cdot;z_i)}(w_t, \beta)} \\
x_{t+1} &= \alpha y_{t+1} + (1-\alpha)x_t
\end{align}
The \ambsgd algorithm will converge on $f^{(\beta)}$ at a rate (see \citep{lan2012optimal,cotter2011better})
\begin{equation}
\mathbb{E}\left[ f(x_T; z) \right] - \min_x \mathbb{E}\left[ f(x;z) \right] = O\parens{\min\set{\frac{L}{\sqrt{T}},\ \frac{\beta}{T^2}} + \frac{L}{\sqrt{MT}}}
\end{equation}
Choosing $\beta = \min\set{LT,H}$ the conclude
\begin{align}
\mathbb{E}\left[ f(x_T; z) \right] - \min_x \mathbb{E}\left[ f(x;z) \right]
&\leq 
\mathbb{E}\left[ f^{(\beta)}(x_T; z) \right] + \frac{L}{2T} - \min_x \mathbb{E}\left[ f^{(\beta)}(x;z) \right]  \\
&= O\parens{\min\set{\frac{L}{\sqrt{T}},\ \frac{\min\set{LT,H}}{T^2}} + \frac{L}{\sqrt{MT}} + \frac{L}{T}} \\
&= O\parens{\min\set{\frac{L}{T},\ \frac{H}{T^2}} + \frac{L}{\sqrt{MT}}}
\end{align}
which matches the lower bound in Theorem \ref{thm:proxLB}.

\subsection{Wait-and-collect accelerated mini-batch SGD}\label{sec:waitandcollect}
\begin{algorithm}[H]
  \caption{"Wait-and-collect" accelerated minibatch SGD}
  \label{alg:wait-and-collect}
  \begin{algorithmic}
    \STATE Initialize $\hat x = \tilde x = x_0=0,$, parameter $\alpha$.
    \FOR{$t=1,2,\dots,T$}
    \IF{$\mod(t,2\tau+1) \leq \tau$}
    \STATE Query stochastic gradient at $\tilde x.$
    \STATE Update $x_t \leftarrow x_{t-1}, \tilde g = 0.$
    \ELSIF{$\mod(t,2\tau+1) > \tau \; \text{\textbf{and}} \mod(t,2\tau+1) \leq 2\tau$}
    \STATE Update $x_t \leftarrow x_{t-1}.$
    \STATE Receive noisy gradient $g_{t-1-\tau}$, update $\tilde g \leftarrow \tilde g + (1/\tau)*g_{t-1-\tau}$
    \ELSIF{$\mod(t,2\tau+1) = 0$}
    \STATE Update $x_t \leftarrow \tilde x - \eta \tilde g$.
    \STATE Update $\hat x \leftarrow \alpha \hat x + (1-\alpha) x_t$, $\tilde x \leftarrow \alpha \hat x + (1-\alpha) x_t$.
    \ENDIF
    \ENDFOR
  \end{algorithmic}
\end{algorithm}

\subsection{Analysis of technical results in Section \ref{sec:intermittent}}
\label{sec:analysis_svrg}

\paragraph{Applying SVRG under intermittent synchronization graph} To apply SVRG method to solve stochastic convex optimization problems under intermittent synchronization graph. We adopt the approach by \citep{lee2017distributed,wang2017memory}, first we sample $n$ instances $\{z_1,...,z_n\}$ and solve a regularized empirical risk minimization problem based on $\{z_1,...,z_n\}$:
\begin{align}
\min_{x} \hat F_{\lambda}(x) := \frac{1}{n} \sum_{i=1}^n f(x;z_i) + \frac{\lambda}{2} \norm{x}^2,  
\label{eqn:erm}
\end{align}
where $\lambda$ is the regularization parameter will specified later. We will apply SVRG algorithm on the intermittent synchronization graph to solve above empirical objective \eqref{eqn:erm} to certain sub-optimality. The SVRG method works in stages, at each stage, we first use $n/KM$ communication rounds to calculate the full gradient of \eqref{eqn:erm} at a reference point $\tilde x$, and then using a single chain to perform stochastic gradient updates, equipped with $\nabla \hat F(\tilde w)$ to reduce the variance. We choose $\lambda \asymp L/(\sqrt{n} B)$, which will makes the objective \eqref{eqn:erm} to be at least $L/(\sqrt{n} B)$-strongly convex, thus the condition number of \eqref{eqn:erm} will be bounded by $O(H/(L/(\sqrt{n} B))) = O(H\sqrt{n}B/L)$. The SVRG analysis \citep{johnson13svrg} requires the number of stochastic gradient updates to be scales as the condition number, so here we will use $O(H\sqrt{n}B/(LK))$ communication rounds to perform the stochastic updates, since one chain within each communication round has length $K$. Let $\hat x^* = \arg\min_{x} \hat F_{\lambda}(x)$, and let $\hat x_s$ to be the iterate after running the SVRG algorithm for $s$-stages. By the standard results of SVRG (Theorem 1 in \citep{johnson13svrg}), we have
\[
\E{ \hat F_{\lambda}(\hat x_s) } - \hat F_{\lambda}(\hat x^*) \leq \left( \frac{1}{2} \right)^s.
\]
By standard estimation-optimization error decomposition (e.g. Section 4 in \citep{shalev2008svm}), we have
\begin{align}
\E{F(\hat x_s)} - F(x^*) \leq& 2 \E{ \hat F_{\lambda}(\hat x_s)  - \hat F_{\lambda}(\hat x^*) } + \frac{\lambda B^2}{2} + O \left( \frac{L^2}{\lambda n} \right) \nonumber \\
\leq& \left( \frac{1}{2} \right)^s + O \left( \frac{LB}{\sqrt{n}} \right) = O \left( \frac{LB}{\sqrt{n}} \right),
\label{eqn:svrg_opt}
\end{align}
given $s \asymp \log (n/(LB))$. Thus to implement SVRG successfully, we need to choose $n$ such that the following two conditions are satisfied:
\begin{align*}
\frac{n}{KM}*s \leq T, \quad \text{and} \quad \frac{H\sqrt{n}B}{LK}*s \leq T.
\end{align*}
Thus we know by choosing $n$ below will satisfy above condition:
\[
n \asymp \min \left\{ \frac{K^2 T^2 L^2}{H^2 B^2 \log^2(MKT/L)}, \frac{M K T}{\log(MKT/L)} \right\},
\]
substitute the scale of $n$ to \eqref{eqn:svrg_opt} we get
\begin{align*}
\E{F(\hat x_s)} - F(x^*) \leq& O \left( \frac{HB^2}{KT} \log \left( \frac{MKT}{L} \right) + \frac{LB}{\sqrt{MKT}} \left( \log \left( \frac{MKT}{L} \right) \right)^{1/2} \right) \\
\leq& O \left( \left( \frac{HB^2}{KT}  + \frac{LB}{\sqrt{MKT}} \right) \log \left( \frac{MKT}{L} \right) \right) , 
\end{align*}
and we obtain the desired result.

\end{document}